\documentclass{article}
\usepackage[utf8]{inputenc}
\usepackage{graphicx}
\usepackage{psfrag}
\usepackage{epsf}
\usepackage{amsmath,amsfonts,amssymb,latexsym,amsthm, comment,amsthm}
\usepackage{algorithmic}
\usepackage{algorithm}
\usepackage[margin=1.1in]{geometry}
\usepackage{setspace}
\usepackage{enumerate}
\usepackage{tikz}
\usepackage{soul}
\usepackage{xcolor}
\usepackage{hyperref}

\usetikzlibrary{calc}
\tikzstyle{line}=[draw]
\usetikzlibrary{shapes.geometric}

\newcommand{\for}{\textrm{for}\,\,}
\newcommand{\ignore}[1]{}
\newcommand{\startClaims}{\setcounter{claim}{0}}
\newtheorem{theorem}{Theorem}[section]
\newtheorem{corollary}[theorem]{Corollary}
\newtheorem{lemma}[theorem]{Lemma}
\newtheorem{proposition}[theorem]{Proposition}

\newtheorem{problem}[theorem]{Problem}

\newtheorem*{corollary*}{Corollary}

\theoremstyle{plain}

\newcommand{\eopf}{\raisebox{0.8ex}{\framebox{}}}

{\noindent{\bf Sketch of a proof.}\ }%
{\hfill\eopf\par\bigskip}%
{\noindent{\bf Ideas for a proof.}\ }%
{\hfill\eopf\par\bigskip}%
{\noindent{\bf Proof in progress.}\ }%
{\hfill\eopf\par\bigskip}%

\newcommand{\diag}{\operatorname{diag}}

\newcommand{\jdup}{\operatorname{jdup}}

\title{Regular Graphs of Degree at most Four that Allow Two Distinct Eigenvalues}

\author{Wayne Barrett
\thanks{Department of Mathematics, Brigham Young University, Provo, UT (wb@mathematics.byu.edu) }
\and Shaun Fallat
\thanks{Department of Mathematics and Statistics, University of Regina, Regina, Saskatchewan, CA (shaun.fallat@uregina.ca)} ~\thanks{Corresponding Author}
\and Veronika Furst \thanks{Department of Mathematics, Fort Lewis College, Durango, CO (furst\_v@fortlewis.edu)}
\and Shahla Nasserasr \thanks{School of Mathematical Sciences, Rochester Institute of Technology, Rochester, NY (sxnsma@rit.edu)}
\and Brendan Rooney \thanks{School of Mathematical Sciences, Rochester Institute of Technology, Rochester, NY (brsma@rit.edu)} 
\and Michael Tait \thanks{Department of Mathematics \& Statistics, Villanova University, Villanova, PA (michael.tait@villanova.edu)}}

\date{\today}

\numberwithin{equation}{subsection}

\begin{document}
\maketitle

\begin{abstract}

For an $n \times n$ matrix $A$, let $q(A)$ be the number of distinct eigenvalues of $A$.  If $G$ is a connected graph on $n$ vertices, let $\mathcal{S}(G)$ be the set of all real symmetric $n \times n$ matrices $A=[a_{ij}]$ such that for $i\neq j$, $a_{ij}=0$ if and only if $\{i,j\}$ is not an edge of $G$.  Let $q(G)={\rm min}\{q(A)\,:\,A \in \mathcal{S}(G)\}$.  Studying $q(G)$ has become a fundamental sub-problem of the inverse eigenvalue problem for graphs, and characterizing the case for which $q(G)=2$ has been especially difficult.  This paper considers the problem of determining the regular graphs $G$ that satisfy $q(G)=2$. The resolution is straightforward if the degree of regularity is $1, 2,$ or $3$.  However, the $4$-regular graphs with $q(G)=2$ are much more difficult to characterize.  A connected $4$-regular graph has $q(G)=2$ if and only if either $G$ belongs to a specific infinite class of graphs, or else $G$ is one of fifteen $4$-regular graphs whose number of vertices ranges from $5$ to $16$.  This technical result gives rise to several intriguing questions. \\

\noindent {\bf Keywords} inverse eigenvalue problem for graphs, orthogonality, $q$-parameter, regular graphs. \\

\noindent {\bf AMS subject classification} 05C50, 15A29, 15A18. \\


\end{abstract}

\section{Introduction}\label{Introduction}

For any connected graph $G$ on $n$ vertices, let $\mathcal{S}(G)$ denote the set of all real symmetric $n\times n$ matrices $A=[a_{ij}]$ where $a_{ij}=0$ if and only if $\{i,j\}$ is not an edge in $G$, and the entries $a_{ii}$ can take any value. The \emph{inverse eigenvalue problem} for a graph $G$ asks to determine all possible spectra of matrices in $\mathcal{S}(G)$ \cite{MR4074182, MR4478249}. 

This problem and several of its sub-problems have been studied extensively. 
One of these sub-problems is to consider all possible multiplicity lists of eigenvalues of matrices in $\mathcal{S}(G)$. 
If we look at the multiplicity lists of eigenvalues of all matrices in $\mathcal{S}(G)$ as lists of numbers, the shortest length among all these lists is the minimum number of distinct eigenvalues of matrices in $\mathcal{S}(G)$. 
This parameter is denoted by $q(G)$ and has been studied in \cite{MR3118943, MR3665573, MR1899084, MR3891770, levene2020orthogonal}. In this paper, we investigate the problem of determining which regular connected graphs $G$ have a matrix in $\mathcal{S}(G)$ with exactly two distinct eigenvalues, that is, with $q(G)=2$.

The connected graphs $G$ with $q(G) = n-1$ or $n$ have been characterized, see \cite{MR3665573}.  Graphs with $q(G)=2$ are much harder to describe; for example, there is no forbidden subgraph characterization of graphs with $q(G)=2$, as implied by Theorem 5.2 in \cite{MR3118943}. It is known that $q(G)=2$ if and only if there is an orthogonal matrix in $\mathcal{S}(G)$ \cite{MR3118943}, and so studying graphs $G$ on $n$ vertices with $q(G)=2$ is equivalent to studying all possible zero patterns of $n\times n$ symmetric orthogonal matrices. 

A graph $G$ must have a sufficiently large number of edges to satisfy $q(G)=2$. In \cite{AllowsProblem}, we showed that a connected graph $G$ on $n$ vertices with $q(G)=2$ has at least  $2n - 4$ edges. We also characterized the graphs for which equality is attained. This result immediately implies that the number of $r$-regular graphs with $r\in \{2,3\}$ is finite, and in Section \ref{Regular} we characterize these graphs. When considering $4$-regular graphs with $q(G)=2$, the difficulty increases significantly. Our main theorem (Theorem \ref{Regular4}) characterizes all connected $4$-regular graphs with $q(G)=2$.

Throughout this paper, we only consider connected, simple, undirected graphs. 

\subsection{Preliminaries}\label{Preliminaries}

One of the common ways to give a lower bound on $q(G)$ is to find a unique shortest path between vertices. This technique, specialized to the case $q(G)=2$, is explained in the following lemma, which is a corollary of Theorem 3.2 in \cite{MR3118943}.

\begin{lemma}\label{shortest paths lemma}
Let $G$ be a connected graph with $q(G)=2$. If $xuy$ is a path of length $2$, then either $x\sim y$ or there is another path $xvy$ of length $2$ between $x$ and $y$.
\end{lemma}

In \cite{AllowsProblem}, we used this lemma extensively in combination with a breadth-first search of a graph, and we use this strategy here as well. Given a fixed vertex $v$, we perform a breadth-first search from $v$. Denote the vertices at distance exactly $i$ from $v$ by $N_i(v)$, and call this set the $i$th {\em distance set} from $v$. We use $\epsilon(v)$ to denote the {\em eccentricity} of $v$, which is the maximum distance from $v$ to any vertex in the graph. The distance sets from $v$ partition the vertex set of $G$ as 
\[
V(G) = \bigcup_{i=0}^{\epsilon(v)} N_i(v),
\]
which we call the \emph{distance partition} of $G$ with respect to $v$. If $u\in N_i(v)$ and $w\in N_{i+1}(v)$ and $u\sim w$, we call $u$ a {\em predecessor} of $w$ and we call $w$ a {\em successor} of $u$. A {\em terminal vertex} is a vertex with no successors.

Assume that $G$ is a graph with $q(G)=2$ and consider the distance partition from a vertex $v$. If a vertex $u$ is in $N_i(v)$ for some $i\geq 2$, then by Lemma \ref{shortest paths lemma}, it must have at least two predecessors, as otherwise there would be a unique shortest path of length $2$ from $u$ to a vertex in $N_{i-2}(v)$. If $G$ is a $4$-regular graph on $n$ vertices, then there are $n-5$ vertices in 
\[
\bigcup_{i=2}^{\epsilon(v)} N_i(v) = V(G) \setminus (\{v\} \cup N_1(v)),
\]
and each of these vertices has at least two predecessors so there are at least $2(n-5)$ edges incident to these vertices. With the 4 edges incident to $v$, this accounts for $2n-6$ of the $2n$ edges of $G$. 
We call the remaining six edges {\em extra edges}, and throughout the paper we consider the possible locations of these six extra edges.

We use standard graph theory terminology and notations.  Often we abbreviate an edge $\{u,v\} \in E(G)$ as $uv$ for vertices $u,v \in V(G)$.  For two graphs $G$ and $H$ with vertex sets $V(G)$ and $V(H)$, respectively, the Cartesian product $G\Box H$ is the graph with vertex set $V(G)\times V(H)$ and $(g_1,h_1)$ adjacent to $(g_2,h_2)$ if either $g_1=g_2$ and $h_1h_2\in E(H)$, or $h_1=h_2$ and $g_1g_2\in E(G)$. The complete graph on $n$ vertices, complete bipartite graph on partite sets of sizes $m$ and $n$, the cycle on $n$ vertices, the path on $n$ vertices, and the hypercube graph (the graph on $2^n$ vertices obtained by an $n$-fold Cartesian product of $K_2$ with itself) are denoted by $K_n$, $K_{m,n}$, $C_n$, $P_n$, $Q_n$, respectively. The circulant graph $G=C(n,\pm i,\pm j)$ is the graph with vertex set $V(C(n,\pm i,\pm j))=\mathbb{Z}/n\mathbb{Z}$ that has edges $\{t,t\pm i\}$ and $\{t,t\pm j\}$ for all $t\in V(G)$. 

Let $G$ be a graph with $v\in V(G)$. A graph $\jdup(G,v)$ is constructed from $G$ by \emph{joined duplicating} a vertex $v\in V(G)$ if $V(\jdup(G,v))=V(G)\cup\{u\}$ and $E(\jdup(G,v))=E(G)\cup\{uw\,:\, w\in \{v\}\cup N_1(v)\}$. From Lemma 2.9 in \cite{MR3891770}, $q(\jdup(G,v))\leq q(G)$ for any vertex $v$ in a connected graph $G$.

\begin{lemma}\label{maxindependnet}[Lemma 2.3, \cite{MR4044603}]
Let $G$ be a connected graph on $n$ vertices with $q(G)=2$. If $S$ is an independent set of vertices, then $|S|\leq k$ where $k$ is the least integer such that there is a matrix in $\mathcal{S}(G)$ with two distinct eigenvalues of multiplicities $k$ and $n-k$. 
\end{lemma}

\begin{lemma}\label{OrthLemma}
Let
\[
M=\left[\begin{array}{cc}
C & B\\
B^T & D
\end{array}\right]
\]
where $C$ is an $m\times m$ symmetric matrix, $B=[b_1\ldots\ b_n]$ has no zero columns, and $D=\diag(d_1,\ldots,d_n)$. If $M$ is orthogonal, then $n\leq m$, the vectors $b_1,\ldots, b_n$ are pairwise orthogonal, and each $b_i$ is a $(-d_i)$-eigenvector for $C$.
\end{lemma}
\begin{proof}
Consider the last $n$ columns of $M$. Since $D$ is diagonal, these columns are pairwise orthogonal if and only if $n\leq m$ and the vectors $b_1,\ldots, b_n$ are pairwise orthogonal. 

Second, expanding $M^2=I$ we have
\[
M^2=\left[\begin{array}{rr}
C^2+BB^T & CB+BD\\
B^TC+DB^T & B^TB+D^2\end{array}\right]=\left[\begin{array}{cc}
I_m & 0\\
0 & I_n\end{array}\right].
\]
From the $(1,2)$-block we see $CB+BD=0$. Rearranging this equation we have
\[
[Cb_1\ldots Cb_n]=-[d_1b_1\ldots d_nb_n]
\]
from which it follows that each $b_i$ is a $(-d_i)$-eigenvector for $C$.
\end{proof}
As an illustration of Lemma \ref{OrthLemma}, suppose $G$ is a connected bipartite graph with bipartition $V=V_1 \cup V_2$. Further assume that $|V_1|=|V_2|$. It follows that if $q(G)=2$, then there exists a matrix $M \in S(G)$, where 
\[
M=\left[\begin{array}{cc}
C & B\\
B^T & D
\end{array}\right],
\] where $M^2=I$, and both $C$ and $D$ are diagonal matrices. Applying Lemma \ref{OrthLemma} we have that $B$ must be a matrix with orthogonal rows and columns. In fact, the converse also holds in this case. If such an orthogonal matrix $B$ exists, then the matrix
\[
M=\frac{1}{\sqrt{2}}\left[\begin{array}{cc}
I & B\\
B^T & -I
\end{array}\right]
\] is orthogonal and hence $q(G)=2$.
%

\section{Certain Regular Graphs that Allow Two Distinct Eigenvalues}\label{Regular}

From Theorem 3.1 in \cite{AllowsProblem}, if a connected graph $G$ has fewer than $2n-4$ edges, then $q(G)>2$. This implies that there are only finitely many $r$-regular graphs with $r\leq 3$ that satisfy $q(G)=2$. We describe them below.
\begin{lemma}\label{RegularN}
Let $r\leq 3$. If $G$ is a connected $r$-regular graph with $q(G)=2$, then $G$ has at most $8/(4-r)$ vertices.
\end{lemma}
\begin{proof}
In order for $G$ to have $q(G)=2$, $G$ has to have at least $2n-4$ edges. So we have the inequality
\[
(r/2)n\geq 2n-4.
\]
Since $0<r/2<2$, this simplifies to $n\leq 8/(4-r)$.
\end{proof}
\begin{corollary}\label{Regular123}
If $G$ is a connected $r$-regular graph with $q(G)=2$ for some $r\leq 3$, then $G$ is one of:
\begin{enumerate}[(1)]
\item $K_2$;
\item $K_3$ or $C_4$; or,
\item $K_4$, $K_{3,3}$, $K_3\Box K_2$, or $Q_3$.
\end{enumerate}
\end{corollary}
\begin{proof}
The graph $K_2$ is the only connected $1$-regular graph and has $q(K_2)=2$.  The only connected $2$-regular graphs on $n \le 4$ vertices are $K_3$ and $C_4$ and both have $q(G)=2$.  By Lemma \ref{RegularN}, a $3$-regular graph with $q(G)=2$ has $4, 6$ or $8$ vertices.  If $n=4$, $G=K_4$ and $q(G)=2$. If $n=6$, the complement of $G$ is $2$-regular and so must be $C_6$ or $2K_3$. Thus $G$ is either $K_3\Box K_2$ or $K_{3,3}$, respectively, both of which have $q(G)=2$ from Corollaries 6.5 and 6.8 in \cite{MR3118943}.  If $n=8$, $G$ has $12 = 2(8)-4$ edges.  Thus by Theorem 3.1 from \cite{AllowsProblem}, $G\cong Q_3$.
%
\end{proof}

We now proceed with the main purpose of this paper, to characterize the $4$-regular graphs $G$ with $q(G) = 2$. We begin by defining an infinite family of graphs called closed candles which are analogs to the single-ended and double-ended candles in \cite{AllowsProblem}. For $k\geq 3$ the \emph{closed candle}, $H_k$, is constructed from $2C_k$ as follows.  Label the vertices of one $C_k$ with the odd integers from 1 to $2k-1$ and the other with the even integers from 2 to $2k$.  Insert $2k$ additional edges between the two $C_k$'s according to the rule: $i$ is adjacent to $j$, $i$ odd, $j$ even if $j-i = 3$, $j-i = -1$, or $j=2$, $i=2k-1$, or $i=1$, $j=2k$. Thus $H_k$ is a $4$-regular graph. The graph $H_{10}$ is shown in Figure \ref{closedCandle}.
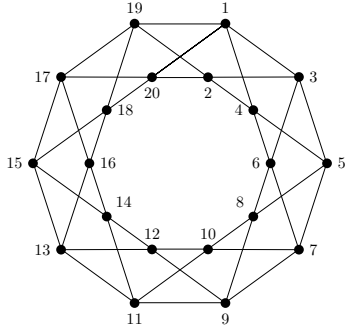
\begin{figure}[ht!]
\centering
\scalebox{.6}{\begin{tikzpicture}[scale=1]
    \node[regular polygon,
    regular polygon sides=10,
    minimum size=6.5cm,
    draw] at (0,0) (B) {};
    \foreach \i [evaluate={\Name=int(\i)}] in {1}
    \node[circle, scale=0.65,
    label=above:{$\Name$},
    fill=black] at (B.corner \i) {};
\foreach \i [evaluate={\Name=int(-2*\i+23)}] in {3,4,5}
    \node[circle, scale=0.65,
    label=left:{$\Name$},
    fill=black] at (B.corner \i) {};
    \foreach \i [evaluate={\Name=int(-2*\i+23)}] in {8,9,10}
    \node[circle, scale=0.65,
    label=right:{$\Name$},
    fill=black] at (B.corner \i) {};
     \foreach \i [evaluate={\Name=int(-2*\i+23)}] in {2}
    \node[circle, scale=0.65,
    label=above:{$\Name$},
    fill=black] at (B.corner \i) {};
     \foreach \i [evaluate={\Name=int(-2*\i+23)}] in {6,7}
    \node[circle, scale=0.65,
    label=below:{$\Name$},
    fill=black] at (B.corner \i) {};
\node[regular polygon,
    regular polygon sides=10,
    minimum size=4cm,
    draw] at (0,0) (A) {};
    \foreach \i [evaluate={\Name=int(\i+1)}] in {1}
    \node[circle, scale=0.65,
    label=below:{$\Name$},
    fill=black] at (A.corner \i) {};
\foreach \i [evaluate={\Name=int(-2*\i+24)}] in {3,4}
    \node[circle, scale=0.65,
    label=right:{$\Name$},
    fill=black] at (A.corner \i) {};
    \foreach \i [evaluate={\Name=int(-2*\i+24)}] in {6,7}
    \node[circle, scale=0.65,
    label=above:{$\Name$},
    fill=black] at (A.corner \i) {};
     \foreach \i [evaluate={\Name=int(-2*\i+24)}] in {5}
    \node[circle, scale=0.65,
    label=above right:{$\Name$},
    fill=black] at (A.corner \i) {};
     \foreach \i [evaluate={\Name=int(-2*\i+24)}] in {2}
    \node[circle, scale=0.65,
    label=below:{$\Name$},
    fill=black] at (A.corner \i) {};
    \foreach \i [evaluate={\Name=int(-2*\i+24)}] in {9,10}
    \node[circle, scale=0.65,
    label=left:{$\Name$},
    fill=black] at (A.corner \i) {};
    \foreach \i [evaluate={\Name=int(-2*\i+24)}] in {8}
    \node[circle, scale=0.65,
    label=above left:{$\Name$},
    fill=black] at (A.corner \i) {};
\draw (A.corner 10)--(B.corner 9)--(A.corner 8)--(B.corner 7)--(A.corner 6)--(B.corner 5)--(A.corner 4)--(B.corner 3)--(A.corner 2)--(B.corner 1)--(A.corner 2)--(B.corner 1)--(A.corner 10);
\draw (A.corner 1)--(B.corner 2)--(A.corner 3)--(B.corner 4)--(A.corner 5)--(B.corner 6)--(A.corner 7)--(B.corner 8)--(A.corner 9)--(B.corner 10)--(A.corner 1);
\end{tikzpicture}}
\caption{Closed candle $H_{10}$.}\label{closedCandle}
\end{figure}

In the proof of Theorem \ref{Regular4}, we will consider induced subgraphs that have the same structure as a closed candle. A \emph{candle section} is a graph with vertices $u_1,\ldots,u_t, v_1,\ldots,v_t$ with edges $u_iu_{i+1}$, $v_iv_{i+1}$, $u_{i+1}v_i$, and $u_iv_{i+1}$ for $1\leq i\leq t-1$, see Figure \ref{CandleSection}.
\begin{figure}[ht!]
\centering
\begin{tikzpicture}[scale=0.5, vrtx/.style args = {#1/#2}{%
      circle, draw, fill=black, inner sep=0pt,
      minimum size=6pt, label=#1:#2}]
\node (1) [vrtx=above/$u_1$] at (2,1) {};
\node (2) [vrtx=above/$u_2$] at (4,1) {};
\node (3) [vrtx=above/$u_3$] at (6,1) {};
\node (k-3) [vrtx=above/$u_{t-2}$] at (10,1) {};
\node (k-2) [vrtx=above/$u_{t-1}$] at (12,1)  {};
\node (k-1) [vrtx=above/$u_{t}$] at (14,1)  {};
\node (k+1) [vrtx=below/$v_t$] at (14,-1) {};
\node (k+2) [vrtx=below/$v_{t-1}$] at (12,-1)  {};
\node (k+3) [vrtx=below/$v_{t-2}$] at (10,-1)  {};
\node (2k-1) [vrtx=below/$v_3$] at (6,-1) {};
\node (2k) [vrtx=below/$v_2$] at (4,-1)  {};
\node (2k+1) [vrtx=below/$v_1$] at (2,-1)  {};
\node (dotsu) at (8,1) {$\ldots$};
\node (dotsc) at (8,0) {$\ldots$};
\node (dotsl) at (8,-1) {$\ldots$};
\draw (1) edge (2);
\draw (2) edge (3);
\draw (k-3) edge (k-2);
\draw (k-2) edge (k-1);
\draw (k+1) edge (k+2);
\draw (k+2) edge (k+3);
\draw (2k-1) edge (2k);
\draw (2k) edge (2k+1);
\draw (1) edge (2k);
\draw (2) edge (2k+1);
\draw (2) edge (2k-1);
\draw (3) edge (2k);
\draw (k-3) edge (k+2);
\draw (k-2) edge (k+3);
\draw (k-2) edge (k+1);
\draw (k-1) edge (k+2);
\end{tikzpicture}
\caption{Candle section.}\label{CandleSection}
\end{figure}
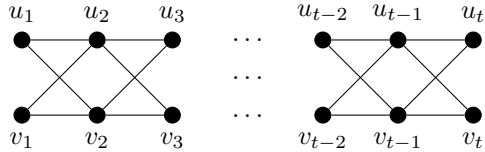

The following lemma gives a construction of orthogonal matrices for the closed candles.

\begin{lemma}\label{qCandle}
For all $k\geq 3$, we have $q(H_k)=2$.
\end{lemma}
\begin{proof}

To see that the graphs $H_k$ for $k\geq 3$ achieve two distinct eigenvalues, we construct orthogonal matrices for this family of graphs.  

Let 
\begin{eqnarray*}
R=\left[ \begin{array}{rr} 
-1 & 1 \\
1 & -1 \end{array} 
\right], \ 
S= \left[ 
\begin{array}{rr}
1 & 1 \\
-1 & -1 \end{array}
\right], \ 
J= \left[ \begin{array}{rr}
\phantom{-}1 & 1 \\
1 & \phantom{-}1 \end{array} \right], \ \textrm{and} \ 
O= \left[ \begin{array}{rr}
0 & 0 \\
0 & 0 \end{array} \right].
\end{eqnarray*}

We consider two cases for $k$ and construct the corresponding matrices $W$. Here all of the blocks are $2\times 2.$ The matrix in each case is symmetric so the blocks below the diagonal blocks are transposes of the corresponding matrices. \\

\noindent\emph{Case 1:} $n=2k$ and $k\geq 4$ is even.

\begin{equation*}
W_{ij}=
    \begin{cases}
      R ,& \for (i,j)= (1,2),(3,4), \ldots,(k-3,k-2),(k-1,k) \ \\
      J,& \for (i,j)=(2,3),(4,5), \ldots,(k-2,k-1) \\
      J ,& \for (i,j)=(1,k)\\
      O,& \textrm{otherwise}\\
    \end{cases}
\end{equation*}

\noindent\emph{Case 2:} $n=2k$ and $k\geq 3$ is odd.

First, note that when $k=3$, the graph $H_3$ is the octahedron (the graph obtained from $K_6$ by deleting a perfect matching). By Corollary 6.9 in \cite{MR3904092}, we have $q(H_3)=q(G204)=2$. Now if $k\geq 5$, we construct the matrices as follows. 

\begin{equation*}
W_{ij}=
    \begin{cases}
      J ,& \for (i,j)= (1,2),(3,4), \ldots, (k-4,k-3) \ \\
      R,& \for (i,j)=(2,3), (4,5), \ldots, (k-5,k-4),\ \ {\rm and} \  (k-2,k-1);\\
      S ,& \for (i,j)=(1,k), (k-3,k-2)\\
      S^{T} ,& \for (i,j)=(k-1,k)\\
      O,& \textrm{otherwise}\\
    \end{cases}
\end{equation*}

Each row of $W$ has Euclidean length $2$.
Since  each of the $2 \times 4$ matrices
$[J \; S]$, $[J\; R]$, $[S^T\; R]$,   $[S \; S^T]$  
have orthogonal rows, each pair of rows of $W$ coming from the same block 
are orthogonal.  For $i\neq j$, the $(i,j)$-block of $W^2$ is one 
of $JR$, $JS$, $SR$,  $S^2$ or their transposes.
In each of the cases the matrix is the zero matrix.  
Hence $W^T W=W^2= 4I$.  We conclude that $\frac{1}{2} W$ is an orthogonal matrix with whose graph 
is the closed candle on $n=2k$ vertices. 
\end{proof}

For example, the matrices for $k=4$ and $k=5$ are respectively

\[
\left[ \begin{array}{ccccc}
 O &R &O &J \\
R & O& J &  O \\ 
   O &J & O & R \\
   J & O & R &O \\
\end{array} \right]
\quad
\mbox{and} 
\quad
\left[ \begin{array}{ccccc}
 O &J &O &O  & S \\
J & O& S & O & O \\
\\
[-10pt]
   O &S^T & O & R & O \\
   O & O & R &O & S^T\\
S^T   & O & O &S & O
\end{array} \right];
\]
for $k=6$ and $k=7$, the matrices are respectively
\[
\left[ \begin{array}{cccccc}
 O &R &O &O &O & J \\
R & O& J & O & O & O\\
O & J & O & R & O & O \\
 O & O & R& O& J & O \\ 
 O & O & O &J & O & R\\
 J & O & O & O & R &O \\
\end{array} \right]
\quad
\mbox{and}
\quad
\left[ \begin{array}{ccccccc}
 O &J &O &O &O & O & S \\
J & O& R & O & O & O & O\\
O & R & O & J & O&O & O \\
 O & O & J& O & S & O &O \\
 \\ 
 [-10pt]
 O & O & O &S^T & O & R & O \\
 O & O & O & O & R &O & S^T\\
S^T  & O & O & O & O &S & O
\end{array} \right].
\]

Note that the closed candle $H_k$ has independence number $k$ if $k$ is even, and $k-1$ if $k$ is odd. By Lemma \ref{maxindependnet} when $k$ is even the only achievable multiplicity list for two distinct eigenvalues is $[k,k]$; when $k$ is odd, the only achievable multiplicity lists for two distinct eigenvalues are $[k,k]$ and $[k-1,k+1]$.

%


Lemma \ref{qCandle} provides an infinite family of $4$-regular graphs with $q(G)=2$. Our main theorem below characterizes all $4$-regular graphs $G$ for which $q(G)=2$.  

\begin{theorem}\label{Regular4}
If $G$ is a connected 4-regular graph with $q(G)=2$, then $G$ is either:
\begin{enumerate}[(1)]
\item $K_5$;
\item one of the graphs $R_{7,1}$, $R_{8,2}$, $R_{8,3}$, $R_{8,4}$, $R_{8,5}$, $R_{8,6}$ from Figure \ref{RGraphs};
\item $K_3\Box C_4$, $K_{3,3}\Box K_2$, one of the graphs $R_{10,2}$, $R_{10,3}$, $R_{10,4}$, $R_{12,3}$, $R_{14,1}$ from Figure \ref{RGraphs2};
\item $Q_4$; or,
\item a closed candle $H_k$ for some $k \ge 3$. 
\end{enumerate}
\end{theorem}
\noindent\emph{Notes:} The graphs listed in items (1) through (4) of Theorem \ref{Regular4} have diameter 1 through 4 respectively. The graph $R_{7,1}\cong C(7,\pm1,\pm2)$, the graph $R_{8,5}\cong K_4\Box K_2$, and the graph $R_{8,6}\cong C(8,\pm1,\pm2)$. The graph $R_{10,3}$ is the graph obtained from $Q_3$ by joined duplicating a pair of antipodal vertices (in \cite{AllowsProblem} we referred to this graph as $Q_3^{\prime}$). The graph $R_{10,4}\cong C(10, \pm 1, \pm 3)$, and the graph $R_{12,3}\cong C(12, \pm1, \pm3)$. The graph $R_{14,1}$ appears in \cite{MR2360149} as $S_{14}$. Moreover $R_{14,1}$ is the Cayley graph for the dihedral group $D_7$ (with generators $\rho$ and $\varphi$ satisfying $\rho^2=\varphi^7=\varepsilon$ and $\rho\varphi=\varphi^6\rho$) with connection set $\{\rho,\varphi\rho,\varphi^2\rho,\varphi^4\rho\}$. It is also the point-block incidence graph of the non-trivial square $2-(7,4,2)$ design and is distance regular with diameter $3$, see \cite{BCN}. The sporadic graphs $R_{6,1}$, $R_{8,1}$, and $R_{10,1}$ in Figure \ref{RGraphs} are $H_3$, $H_4$, and $H_5$, respectively, so appear in item (5) of Theorem \ref{Regular4}. Also note that the closed candles are all circulants. For $k\geq 3$, it can be verified that $H_k\cong C(2k,\pm1,\pm (k-1))$.



The proof of Theorem \ref{Regular4} is split over Sections \ref{SmallDiameter}, \ref{SporadicSection}, and \ref{LargeDiameter}.

\section{Proof of Theorem \ref{Regular4} for 4-Regular Graphs with Small Diameter}\label{SmallDiameter}

In this section we prove items (1) and (2) of Theorem \ref{Regular4} for graphs with diameter at most $2$. The only $4$-regular graph with diameter $1$ is $K_5$, and $q(K_5)=2$, so we focus on $4$-regular graphs with diameter $2$. We begin by enumerating the $4$-regular graphs $G$ with diameter $2$ for which $q(G)=2$ is not ruled out by Lemma \ref{shortest paths lemma}.
\begin{lemma}
If $G$ is a connected $4$-regular graph with diameter $2$ such that $q(G)=2$, then $6\leq |V(G)|\leq 10$.
\end{lemma}
\begin{proof}
Let $G$ be a connected $4$-regular graph with diameter $2$. Consider an arbitrary vertex $v$ in $G$ and the distance partition of $V(G)$ from $v$. In order for $G$ to have $q(G)=2$, it must be the case that each $x\in N_2(v)$ has at least two neighbors in $N_1(v)$, otherwise there is a unique path of length 2 between $x$ and $v$. Let $X$ be the set of edges between $N_1(v)$ and $N_2(v)$. Then 
\[
2|N_2(v)|\leq|X|\leq 3|N_1(v)|=12.
\]
So $|N_2(v)|\leq 6$ and $G$ has at most $1+4+6=11$ vertices. This establishes $6\leq |V(G)|\leq 11$. 

We now show the upper bound can be improved to $10$. Consider a $4$-regular graph $G$ with $11$ vertices and diameter $2$. Using the notation above, we see $|N_2(v)|=6$, and $|X|=12$. So every vertex in $N_1(v)$ has exactly three neighbors in $N_2(v)$, and every vertex in $N_2(v)$ has exactly two neighbors in $N_1(v)$. In particular, this means the subgraph $H$ of $G$ induced by $N_2(v)$ is $2$-regular. So we have two cases: either $H$ is a $6$-cycle, or $H$ is the disjoint union of two $3$-cycles.\\

\noindent\emph{Case 1: $H\cong C_6$}
 
Let the vertices of $H$ be $x_1$, $x_2$, $x_3$, $x_4$, $x_5$, $x_6$ in cyclic order. Note that $H$ is bipartite, and in $H$ there is a unique shortest path of length $2$ between any two vertices in the same partite set. Since there can be no unique shortest path of length $2$ in $G$, the edges of $X$ must supply an additional path of length $2$ between every pair of vertices in each partite set. 

Let $N_1(v)=\{v_1,v_2,v_3,v_4\}$. Since $v_ivv_j$ is a path of length $2$, in order for this path not to be unique there must be some $x_k$ that is adjacent to both $v_i$ and $v_j$. In particular, this means we cannot have some $v_i$ whose neighbors are $\{x_1,x_3,x_5\}$ and some $v_j$ whose neighbors are $\{x_2,x_4,x_6\}$. If there is no $v_i$ whose neighbors are $\{x_1,x_3,x_5\}$, then $\{x_1,x_3,x_5\}$ together with the common neighbor of $\{x_1,x_3\}$, the common neighbor of $\{x_3,x_5\}$ and the common neighbor of $\{x_1,x_5\}$ form a $6$-cycle. Without loss of generality, suppose this $6$-cycle is $(v_1,x_1,v_2,x_3,v_3,x_5)$. Then, in order for each pair of vertices in $\{x_2,x_4,x_6\}$ to have a common neighbor, we must have $v_4$ adjacent to each of $\{x_2,x_4,x_6\}$. Thus the edges between $\{v_1,v_2,v_3\}$ and $\{x_2,x_4,x_6\}$ are a perfect matching.

We know $v_1$ is matched to one of $x_2$, $x_4$, or $x_6$. We also know that $v_1$ is already adjacent to $x_1$ and $x_5$. Note that in $H$, $x_1$ is at distance $3$ from $x_4$. So if $v_1$ is matched to $x_4$, then we have a unique shortest path of length $2$, $x_1v_1x_4$. Similarly, $x_5$ is at distance $3$ from $x_2$ in $H$, so if we match $v_1$ to $x_2$, we get another unique shortest path of length $2$. Thus $v_1$ must be matched to $x_6$. A similar argument shows that $v_2$ must be matched to $x_2$, and $v_3$ must be matched to $x_4$. This accounts for all of the edges in $X$. But now we see that $v_1x_1x_2$ is a unique shortest path of length $2$ in $G$.
\\

\noindent\emph{Case 2: $H\cong 2C_3$}

Let the vertices of $H$ be $x_1,x_2,x_3$ and $y_1,y_2,y_3$, where all of the $x_i$'s are adjacent, and all of the $y_i$'s are adjacent. Let $N_1(v)=\{v_1,v_2,v_3,v_4\}$. We know that $x_1$ has $2$ neighbors in $N_1(v)$. Without loss of generality, suppose $x_1$ is adjacent to $v_1$. Then $v_1x_1x_2$ and $v_1x_1x_3$ are paths of length $2$. In order for them not to be unique shortest paths of length $2$, we must have at least one of the edges $v_1x_2$ and $v_1x_3$. Suppose $v_1$ is adjacent to exactly one of $x_2$ and $x_3$. Then $v_1$ is adjacent to exactly one of the $y_i$ vertices, and we have a unique path $v_1y_iy_j$ between $v_1$ and some $y_j$. Thus $v_1$ must be adjacent to both $x_2$ and $x_3$. If $v_2$ is the other neighbor of $x_1$ in $N_1(v)$, then following the same argument as for $v_1$, we must also have edges $v_2x_2$ and $v_2x_3$. This accounts for all edges in $X$ with ends $v_1$ and $v_2$, and all edges in $X$ with ends $x_1$, $x_2$, or $x_3$. Thus the remaining edges in $X$ are all possible edges between $\{v_3,v_4\}$ and $\{y_1,y_2,y_3\}$. But now $v_1vv_3$ is a unique shortest path of length $2$.
\end{proof}

Since there are only $84$ connected $4$-regular graphs with order $6\leq n\leq 10$ \cite{oeis}, we can generate the list of $4$-regular graphs with diameter $2$ for which $q(G)=2$ is not ruled out by Lemma \ref{shortest paths lemma}. We do this by:
\begin{enumerate}[(1)]
\item using nauty's geng function \cite{MR3131381} to generate all connected $4$-regular graphs on $n$ vertices for each $6\leq n\leq 10$;
\item checking the diameter of the graphs generated in (1), and eliminating all with diameter at least 3; then
\item checking the graphs remaining after (2) for any unique shortest paths connecting two vertices at distance $2$ and eliminating those graphs.
\end{enumerate}
At the end of this computation we are left with a set of thirteen $4$-regular graphs of diameter $2$ for which $q(G)=2$ is not ruled out. Figure \ref{RGraphs} gives these thirteen graphs. Lemma \ref{TableLemma} completes the proof of Theorem \ref{Regular4} for graphs with diameter at most $2$.

\begin{figure}[ht!]
\centering
\begin{tikzpicture}[scale=1.4]
	\node (1)  [draw,circle,inner sep=1pt] at (1-3,0) {1};
	\node (2)  [draw,circle,inner sep=1pt] at (0.5-3,0.866) {2};
	\node (3)  [draw,circle,inner sep=1pt] at (-0.5-3,0.866) {3};
	\node (4)  [draw,circle,inner sep=1pt] at (-1-3,0) {4};
	\node (5)  [draw,circle,inner sep=1pt] at (-0.5-3,-0.866) {5};
	\node (6)  [draw,circle,inner sep=1pt] at (0.5-3,-0.866)  {6};
	\node (label1) at (0-3,-1.5) {$R_{6,1}$};
	\draw (1) edge (2);
	\draw (3) edge (2);
	\draw (3) edge (4);
	\draw (5) edge (4);
	\draw (5) edge (6);
	\draw (1) edge (6);
	\draw (3) edge (5);
	\draw (1) edge (5);
	\draw (1) edge (3);
	\draw (4) edge (6);
	\draw (4) edge (2);
	\draw (6) edge (2);
	\node (a1)  [draw,circle,inner sep=1pt] at (1,0) {1};
	\node (a2)  [draw,circle,inner sep=1pt] at (0.623,0.782) {2};
	\node (a3)  [draw,circle,inner sep=1pt] at (-0.223,0.975) {3};
	\node (a4)  [draw,circle,inner sep=1pt] at (-0.901,0.434) {4};
	\node (a5)  [draw,circle,inner sep=1pt] at (-0.901,-0.434) {5};
	\node (a6)  [draw,circle,inner sep=1pt] at (-0.223,-0.975)  {6};
	\node (a7)  [draw,circle,inner sep=1pt] at (0.623,-0.782)  {7};
	\node (label2) at (0,-1.5) {$R_{7,1}$};
	\draw (a1) edge (a2);
	\draw (a2) edge (a3);
	\draw (a3) edge (a4);
	\draw (a4) edge (a5);
	\draw (a5) edge (a6);
	\draw (a6) edge (a7);
	\draw (a7) edge (a1);
	\draw (a1) edge (a3);
	\draw (a2) edge (a4);
	\draw (a3) edge (a5);
	\draw (a4) edge (a6);
	\draw (a5) edge (a7);
	\draw (a6) edge (a1);
	\draw (a7) edge (a2);
	\node (b1)  [draw,circle,inner sep=1pt] at (1+3,0) {7};
	\node (b2)  [draw,circle,inner sep=1pt] at (0.5+3,0.866) {2};
	\node (b3)  [draw,circle,inner sep=1pt] at (-0.5+3,0.866) {1};
	\node (b4)  [draw,circle,inner sep=1pt] at (-1+3,0) {5};
	\node (b5)  [draw,circle,inner sep=1pt] at (-0.5+3,-0.866) {3};
	\node (b6)  [draw,circle,inner sep=1pt] at (0.5+3,-0.866)  {4};
	\node (b7)  [draw,circle,inner sep=1pt] at (0+3,0)  {6};
	\node (label3) at (0+3,-1.5) {$R_{7,2}$};
	\draw (b1) edge (b2);
	\draw (b3) edge (b2);
	\draw (b3) edge (b4);
	\draw (b5) edge (b4);
	\draw (b5) edge (b6);
	\draw (b1) edge (b6);
	\draw (b2) edge (b7);
	\draw (b3) edge (b7);
	\draw (b5) edge (b7);
	\draw (b6) edge (b7);
	\draw (b4) edge (b6);
	\draw (b4) edge (b2);
	\draw (b1) edge (b5);
	\draw (b1) edge (b3);
	\node (c1)  [draw,circle,inner sep=1pt] at (0.283-3,0.283-3) {5};
	\node (c2)  [draw,circle,inner sep=1pt] at (0.707-3,0.707-3) {1};
	\node (c3)  [draw,circle,inner sep=1pt] at (-0.283-3,0.283-3) {6};
	\node (c4)  [draw,circle,inner sep=1pt] at (-0.707-3,0.707-3) {2};
	\node (c5)  [draw,circle,inner sep=1pt] at (-0.283-3,-0.283-3) {7};
	\node (c6)  [draw,circle,inner sep=1pt] at (-0.707-3,-0.707-3)  {3};
	\node (c7)  [draw,circle,inner sep=1pt] at (0.283-3,-0.283-3)  {8};
	\node (c8)  [draw,circle,inner sep=1pt] at (0.707-3,-0.707-3)  {4};
	\node (label4) at (0-3,-1.5-3) {$R_{8,1}$};
	\draw (c1) edge (c3);
	\draw (c5) edge (c3);
	\draw (c5) edge (c7);
	\draw (c1) edge (c7);
	\draw (c2) edge (c4);
	\draw (c6) edge (c4);
	\draw (c6) edge (c8);
	\draw (c2) edge (c8);
	\draw (c1) edge (c4);
	\draw (c1) edge (c8);
	\draw (c3) edge (c2);
	\draw (c3) edge (c6);
	\draw (c5) edge (c4);
	\draw (c5) edge (c8);
	\draw (c7) edge (c2);
	\draw (c7) edge (c6);
	\node (d1)  [draw,circle,inner sep=1pt] at (1,0-3) {1};
	\node (d2)  [draw,circle,inner sep=1pt] at (0.707,0.707-3) {2};
	\node (d3)  [draw,circle,inner sep=1pt] at (0,1-3) {3};
	\node (d4)  [draw,circle,inner sep=1pt] at (-0.707,0.707-3) {4};
	\node (d5)  [draw,circle,inner sep=1pt] at (-1,0-3) {5};
	\node (d6)  [draw,circle,inner sep=1pt] at (-0.707,-0.707-3)  {6};
	\node (d7)  [draw,circle,inner sep=1pt] at (0,-1-3)  {7};
	\node (d8)  [draw,circle,inner sep=1pt] at (0.707,-0.707-3)  {8};
	\node (label5) at (0,-1.5-3) {$R_{8,2}$};
	\draw (d1) edge (d2);
	\draw (d3) edge (d2);
	\draw (d3) edge (d4);
	\draw (d5) edge (d4);
	\draw (d5) edge (d6);
	\draw (d7) edge (d6);
	\draw (d7) edge (d8);
	\draw (d1) edge (d8);
	\draw (d1) edge (d3);
	\draw (d7) edge (d3);
	\draw (d7) edge (d5);
	\draw (d2) edge (d5);
	\draw (d2) edge (d8);
	\draw (d4) edge (d8);
	\draw (d4) edge (d6);
	\draw (d1) edge (d6);
	\node (e1)  [draw,circle,inner sep=1pt] at (0.707+3,0.707-3) {1};
	\node (e2)  [draw,circle,inner sep=1pt] at (-0.707+3,0.707-3) {2};
	\node (e3)  [draw,circle,inner sep=1pt] at (-0.707+3,-0.707-3) {3};
	\node (e4)  [draw,circle,inner sep=1pt] at (0.707+3,-0.707-3) {4};
	\node (e5)  [draw,circle,inner sep=1pt] at (0+3,0.4-3) {5};
	\node (e6)  [draw,circle,inner sep=1pt] at (-0.4+3,0-3)  {6};
	\node (e7)  [draw,circle,inner sep=1pt] at (0+3,-0.4-3)  {7};
	\node (e8)  [draw,circle,inner sep=1pt] at (0.4+3,0-3)  {8};
	\node (label6) at (0+3,-1.5-3) {$R_{8,3}$};
	\draw (e1) edge (e2);
	\draw (e3) edge (e2);
	\draw (e3) edge (e4);
	\draw (e1) edge (e4);
	\draw (e5) edge (e7);
	\draw (e6) edge (e7);
	\draw (e6) edge (e8);
	\draw (e5) edge (e8);
	\draw (e1) edge (e5);
	\draw (e1) edge (e8);
	\draw (e2) edge (e5);
	\draw (e2) edge (e6);
	\draw (e3) edge (e6);
	\draw (e3) edge (e7);
	\draw (e4) edge (e8);
	\draw (e4) edge (e7);
	\node (f1)  [draw,circle,inner sep=1pt] at (1-3,0-6) {1};
	\node (f2)  [draw,circle,inner sep=1pt] at (0.707-3,0.707-6) {2};
	\node (f3)  [draw,circle,inner sep=1pt] at (0-3,1-6) {3};
	\node (f4)  [draw,circle,inner sep=1pt] at (-0.707-3,0.707-6) {4};
	\node (f5)  [draw,circle,inner sep=1pt] at (-1-3,0-6) {5};
	\node (f6)  [draw,circle,inner sep=1pt] at (-0.707-3,-0.707-6)  {6};
	\node (f7)  [draw,circle,inner sep=1pt] at (0-3,-1-6)  {7};
	\node (f8)  [draw,circle,inner sep=1pt] at (0.707-3,-0.707-6)  {8};
	\node (label7) at (0-3,-1.5-6) {$R_{8,4}$};
	\draw (f1) edge (f2);
	\draw (f3) edge (f2);
	\draw (f3) edge (f4);
	\draw (f5) edge (f4);
	\draw (f5) edge (f6);
	\draw (f7) edge (f6);
	\draw (f7) edge (f8);
	\draw (f1) edge (f8);
	\draw (f1) edge (f5);
	\draw (f1) edge (f7);
	\draw (f2) edge (f5);
	\draw (f2) edge (f4);
	\draw (f3) edge (f7);
	\draw (f3) edge (f8);
	\draw (f6) edge (f4);
	\draw (f6) edge (f8);
	\node (g1)  [draw,circle,inner sep=1pt] at (1,0-6) {1};
	\node (g2)  [draw,circle,inner sep=1pt] at (0.707,0.707-6) {2};
	\node (g3)  [draw,circle,inner sep=1pt] at (0,1-6) {3};
	\node (g4)  [draw,circle,inner sep=1pt] at (-0.707,0.707-6) {4};
	\node (g5)  [draw,circle,inner sep=1pt] at (-1,0-6) {5};
	\node (g6)  [draw,circle,inner sep=1pt] at (-0.707,-0.707-6)  {6};
	\node (g7)  [draw,circle,inner sep=1pt] at (0,-1-6)  {7};
	\node (g8)  [draw,circle,inner sep=1pt] at (0.707,-0.707-6)  {8};
	\node (label8) at (0,-1.5-6) {$R_{8,5}$};
	\draw (g1) edge (g2);
	\draw (g3) edge (g2);
	\draw (g3) edge (g4);
	\draw (g5) edge (g4);
	\draw (g5) edge (g6);
	\draw (g7) edge (g6);
	\draw (g7) edge (g8);
	\draw (g1) edge (g8);
	\draw (g1) edge (g4);
	\draw (g1) edge (g3);
	\draw (g4) edge (g2);
	\draw (g5) edge (g8);
	\draw (g6) edge (g8);
	\draw (g5) edge (g7);
	\draw (g3) edge (g6);
	\draw (g7) edge (g2);
	\node (h1)  [draw,circle,inner sep=1pt] at (1+3,0-6) {1};
	\node (h2)  [draw,circle,inner sep=1pt] at (0.707+3,0.707-6) {2};
	\node (h3)  [draw,circle,inner sep=1pt] at (0+3,1-6) {3};
	\node (h4)  [draw,circle,inner sep=1pt] at (-0.707+3,0.707-6) {4};
	\node (h5)  [draw,circle,inner sep=1pt] at (-1+3,0-6) {5};
	\node (h6)  [draw,circle,inner sep=1pt] at (-0.707+3,-0.707-6)  {6};
	\node (h7)  [draw,circle,inner sep=1pt] at (0+3,-1-6)  {7};
	\node (h8)  [draw,circle,inner sep=1pt] at (0.707+3,-0.707-6)  {8};
	\node (label9) at (0+3,-1.5-6) {$R_{8,6}$};
	\draw (h1) edge (h2);
	\draw (h3) edge (h2);
	\draw (h3) edge (h4);
	\draw (h5) edge (h4);
	\draw (h5) edge (h6);
	\draw (h7) edge (h6);
	\draw (h7) edge (h8);
	\draw (h1) edge (h8);
	\draw (h1) edge (h3);
	\draw (h5) edge (h3);
	\draw (h5) edge (h7);
	\draw (h1) edge (h7);
	\draw (h2) edge (h4);
	\draw (h6) edge (h4);
	\draw (h6) edge (h8);
	\draw (h2) edge (h8);
	\node (i1)  [draw,circle,inner sep=1pt] at (1-3,0-9) {1};
	\node (i2)  [draw,circle,inner sep=1pt] at (0.766-3,0.643-9) {2};
	\node (i3)  [draw,circle,inner sep=1pt] at (0.174-3,0.985-9) {3};
	\node (i4)  [draw,circle,inner sep=1pt] at (-0.5-3,0.866-9) {4};
	\node (i5)  [draw,circle,inner sep=1pt] at (-0.94-3,0.342-9) {5};
	\node (i6)  [draw,circle,inner sep=1pt] at (-0.94-3,-0.342-9)  {6};
	\node (i7)  [draw,circle,inner sep=1pt] at (-0.5-3,-0.866-9)  {7};
	\node (i8)  [draw,circle,inner sep=1pt] at (0.174-3,-0.985-9)  {8};
	\node (i9)  [draw,circle,inner sep=1pt] at (0.766-3,-0.643-9)  {9};
	\node (label10) at (0-3,-1.5-9) {$R_{9,1}$};
	\draw (i1) edge (i2);
	\draw (i3) edge (i2);
	\draw (i3) edge (i4);
	\draw (i5) edge (i4);
	\draw (i5) edge (i6);
	\draw (i7) edge (i6);
	\draw (i7) edge (i8);
	\draw (i9) edge (i8);
	\draw (i9) edge (i1);
	\draw (i1) edge (i4);
	\draw (i7) edge (i4);
	\draw (i7) edge (i1);
	\draw (i9) edge (i3);
	\draw (i9) edge (i6);
	\draw (i3) edge (i6);
	\draw (i8) edge (i2);
	\draw (i5) edge (i2);
	\draw (i5) edge (i8);
	\node (j1)  [draw,circle,inner sep=1pt] at (1,0-9) {1};
	\node (j2)  [draw,circle,inner sep=1pt] at (0.766,0.643-9) {2};
	\node (j3)  [draw,circle,inner sep=1pt] at (0.174,0.985-9) {3};
	\node (j4)  [draw,circle,inner sep=1pt] at (-0.5,0.866-9) {4};
	\node (j5)  [draw,circle,inner sep=1pt] at (-0.94,0.342-9) {5};
	\node (j6)  [draw,circle,inner sep=1pt] at (-0.94,-0.342-9)  {6};
	\node (j7)  [draw,circle,inner sep=1pt] at (-0.5,-0.866-9)  {7};
	\node (j8)  [draw,circle,inner sep=1pt] at (0.174,-0.985-9)  {8};
	\node (j9)  [draw,circle,inner sep=1pt] at (0.766,-0.643-9)  {9};
	\node (label11) at (0,-1.5-9) {$R_{9,2}$};
	\draw (j1) edge (j2);
	\draw (j3) edge (j2);
	\draw (j3) edge (j4);
	\draw (j5) edge (j4);
	\draw (j5) edge (j6);
	\draw (j7) edge (j6);
	\draw (j7) edge (j8);
	\draw (j9) edge (j8);
	\draw (j9) edge (j1);
	\draw (j1) edge (j3);
	\draw (j7) edge (j3);
	\draw (j7) edge (j9);
	\draw (j5) edge (j9);
	\draw (j5) edge (j1);
	\draw (j8) edge (j2);
	\draw (j8) edge (j4);
	\draw (j6) edge (j4);
	\draw (j6) edge (j2);
	\node (k1)  [draw,circle,inner sep=1pt] at (1+3,0-9) {1};
	\node (k2)  [draw,circle,inner sep=1pt] at (0.766+3,0.643-9) {2};
	\node (k3)  [draw,circle,inner sep=1pt] at (0.174+3,0.985-9) {3};
	\node (k4)  [draw,circle,inner sep=1pt] at (-0.5+3,0.866-9) {4};
	\node (k5)  [draw,circle,inner sep=1pt] at (-0.94+3,0.342-9) {5};
	\node (k6)  [draw,circle,inner sep=1pt] at (-0.94+3,-0.342-9)  {6};
	\node (k7)  [draw,circle,inner sep=1pt] at (-0.5+3,-0.866-9)  {7};
	\node (k8)  [draw,circle,inner sep=1pt] at (0.174+3,-0.985-9)  {8};
	\node (k9)  [draw,circle,inner sep=1pt] at (0.766+3,-0.643-9)  {9};
	\node (label12) at (0+3,-1.5-9) {$R_{9,3}$};
	\draw (k1) edge (k2);
	\draw (k3) edge (k2);
	\draw (k3) edge (k4);
	\draw (k5) edge (k4);
	\draw (k5) edge (k6);
	\draw (k7) edge (k6);
	\draw (k7) edge (k8);
	\draw (k9) edge (k8);
	\draw (k9) edge (k1);
	\draw (k1) edge (k3);
	\draw (k7) edge (k3);
	\draw (k7) edge (k4);
	\draw (k9) edge (k4);
	\draw (k9) edge (k5);
	\draw (k2) edge (k5);
	\draw (k2) edge (k6);
	\draw (k8) edge (k6);
	\draw (k8) edge (k1);
	\node (l1)  [draw,circle,inner sep=1pt] at (1,0-12) {1};
	\node (l2)  [draw,circle,inner sep=1pt] at (0.309,0.951-12) {2};
	\node (l3)  [draw,circle,inner sep=1pt] at (-0.909,0.851-12) {3};
	\node (l4)  [draw,circle,inner sep=1pt] at (-0.909,-0.851-12) {4};
	\node (l5)  [draw,circle,inner sep=1pt] at (0.309,-0.951-12) {5};
	\node (l6)  [draw,circle,inner sep=1pt] at (0.4,0-12)  {6};
	\node (l7)  [draw,circle,inner sep=1pt] at (0.124,0.38-12)  {7};
	\node (l8)  [draw,circle,inner sep=1pt] at (-0.324,0.235-12)  {8};
	\node (l9)  [draw,circle,inner sep=1pt] at (-0.324,-0.235-12)  {9};
	\node (l10)  [draw,circle,inner sep=1pt] at (0.124,-0.38-12)  {10};
	\node (label13) at (0,-1.5-12) {$R_{10,1}$};
	\draw (l1) edge (l2);
	\draw (l3) edge (l2);
	\draw (l3) edge (l4);
	\draw (l5) edge (l4);
	\draw (l5) edge (l1);
	\draw (l6) edge (l7);
	\draw (l8) edge (l7);
	\draw (l8) edge (l9);
	\draw (l10) edge (l9);
	\draw (l10) edge (l6);
	\draw (l1) edge (l10);
	\draw (l5) edge (l9);
	\draw (l4) edge (l8);
	\draw (l3) edge (l7);
	\draw (l2) edge (l6);
	\draw (l1) edge (l7);
	\draw (l2) edge (l8);
	\draw (l3) edge (l9);
	\draw (l4) edge (l10);
	\draw (l5) edge (l6);
\end{tikzpicture}
\caption{Thirteen $4$-regular graphs with diameter $2$.}\label{RGraphs}
\end{figure}

\begin{lemma}\label{TableLemma}
Table \ref{RGraphTable} lists all thirteen candidate $4$-regular graphs with diameter 2 (shown in Figure \ref{RGraphs}) and includes their $q$-values (or a bound on their $q$-value). 

\begin{table}[!ht]
\begin{center}
\begin{tabular}{r|r}
Graph & $q$-value\\
\hline
$R_{6,1}$ & 2\\
$R_{7,1}$ & 2\\
$R_{7,2}$ & 3\\
$R_{8,1}$ & 2\\
$R_{8,2}$ & 2\\
\end{tabular}
\quad\quad\quad
\begin{tabular}{r|r}
Graph & $q$-value\\
\hline
$R_{8,3}$ & 2\\
$R_{8,4}$ & 2\\
$R_{8,5}$ & 2\\
$R_{8,6}$ & 2\\
 & \\
\end{tabular}
\quad\quad\quad
\begin{tabular}{r|r}
Graph & $q$-value\\
\hline
$R_{9,1}$ & $>2$\\
$R_{9,2}$ & $>2$\\
$R_{9,3}$ & 3\\
$R_{10,1}$ & 2\\
 & \\
\end{tabular}
\caption{The reduced list of thirteen $4$-regular graphs with diameter $2$.}\label{RGraphTable}
\end{center}
\end{table}
\end{lemma}

\begin{proof}
We treat each of the graphs in Table \ref{RGraphTable} separately, in the order they appear in the table.
\\

\noindent $R_{6,1}$: The graph $R_{6,1}$ is a closed candle, $R_{6,1}\cong H_3$. Thus $q(R_{6,1})=2$ by Lemma \ref{qCandle}.
\\

\noindent $R_{7,1}$: The following matrix $M$ is a matrix in $\mathcal{S}(R_{7,1}-\{1,7\})$ (the graph $R_{7,1}$ with the edge $\{1,7\}$ deleted),
\[
M=\frac{1}{6}\left[\begin{array}{rrrrrrr}
3 & \sqrt{6} & -3 & 0 & 0 & 2\sqrt{3} & 0\\
\sqrt{6} & 0 & \sqrt{6} & 2\sqrt{2} & 0 & 0 & -4\\
-3 & \sqrt{6} & -1 & 2\sqrt{3} & 2\sqrt{2} & 0 & 0\\
0 & 2\sqrt{2} & 2\sqrt{3} & -3 & \sqrt{6} & 1 & 0\\
0 & 0 & 2\sqrt{2} & \sqrt{6} & -2 & \sqrt{6} & 2\sqrt{3}\\
2\sqrt{3} & 0 & 0 & 1 & \sqrt{6} & -3 & 2\sqrt{2}\\
0 & -4 & 0 & 0 & 2\sqrt{3} & 2\sqrt{2} & 0\\
\end{array}\right].
\]
The matrix $M$ is orthogonal and has the Strong Spectral Property; see pages 10 and 11 in \cite{MR3665573}. Thus $q(R_{7,1})=2$.
\\

\noindent $R_{7,2}$: Note that $R_{7,2}$ can be constructed from $K_{4,3}$ by adding edges $\{1,2\}$ and $\{3,4\}$. Suppose $M\in\mathcal{S}(R_{7,2})$. We write $M$ as
\[
M=\left[\begin{array}{cc}
C & B\\
B^T & D\end{array}\right]
\]
where $C$ is a $4\times 4$ matrix, $B=[b_1\ b_2\ b_3]$ has no zero entries, and $D=\diag(d_1,d_2,d_3)$. Moreover,
\[
C=\left[\begin{array}{cc}
C_1 & 0\\
0 & C_2\end{array}\right]
\]
where each $C_i\in\mathcal{S}(K_2)$.

Assume $M$ is an orthogonal matrix. 
Using Lemma \ref{OrthLemma}, the columns of $B$ are pairwise orthogonal and $Cb_i=-d_ib_i$ for $i=1,2,3.$
Partition each $b_i$ into vectors $x_i,y_i\in\mathbb{R}^2$. Now
\[
\left[\begin{array}{r}
C_1x_i\\
C_2y_i\end{array}\right] = \left[\begin{array}{cc}
C_1 & 0\\
0 & C_2\end{array}\right]\left[\begin{array}{r}
x_i\\
y_i\end{array}\right] = Cb_i =\left[\begin{array}{r}
-d_ix_i\\
-d_iy_i\end{array}\right],
\]
so each $x_i$ is a $(-d_i)$-eigenvector for $C_1$, and each $y_i$ is a $(-d_i)$-eigenvector for $C_2$.

The matrices $C_1$ and $C_2$ are each $2\times 2$ non-scalar symmetric matrices, so each has two distinct eigenvalues. Thus, $d_1,d_2,d_3$ cannot be all distinct. Without loss of generality, suppose $d_2=d_3$. And since the dimension of the $(-d_2)$-eigenspace of $C_i$ is $1$, the $(-d_2)$-eigenvectors of $C_i$ are scalar multiples of each other. That is, there exist $\alpha,\beta\neq 0$ so that $x_3=\alpha x_2$ and $y_3=\beta y_2$.

Now we consider the $(2,2)$-block of $M^2$. We have
\[
\begin{array}{rcl}
I_3-D^2 &=& B^TB\\
&=&\left[\begin{array}{rr}
x_1^T & y_1^T\\
x_2^T & y_2^T\\
\alpha x_2^T & \beta y_2^T\end{array}\right]\left[\begin{array}{rrr}
x_1 & x_2 & \alpha x_2\\
y_1 & y_2 & \beta y_2\end{array}\right]\\
\\
& =& \left[\begin{array}{rrr}
x_1^Tx_1+y_1^Ty_1 & 0 & 0\\
0 & x_2^Tx_2+y_2^Ty_2 & 0\\
0 & 0 & \alpha^2x_2^Tx_2+\beta^2y_2^Ty_2
\end{array}\right].
\end{array}
\]
From
\[
x_2^Tx_2+y_2^Ty_2 = 1-d_2^2 =  \alpha^2x_2^Tx_2+\beta^2y_2^Ty_2
\]
we conclude
\begin{equation}\label{eq:1}
 (\alpha^2-1)x_2^Tx_2+(\beta^2-1)y_2^Ty_2=0.
\end{equation}
Since the second and third columns of $B$ are orthogonal we also have
\begin{equation*}\label{eq:2}
\alpha x_2^Tx_2+\beta y_2^Ty_2=0,
\end{equation*}
thus 
\begin{equation}\label{eq:3}
x_2^Tx_2=-\displaystyle \frac{\beta}{\alpha} y_2^Ty_2. 
\end{equation}
Substituting \eqref{eq:3} into \eqref{eq:1} we obtain 
\begin{eqnarray*}
 (\alpha^2-1)\left(-\displaystyle \frac{\beta}{\alpha} y_2^Ty_2\right)+(\beta^2-1)y_2^Ty_2=0 &\Rightarrow& -\alpha^2\beta+\beta+\alpha\beta^2 -\alpha=0\\
 &\Rightarrow&(\beta-\alpha)(\alpha\beta+1)=0 \\
 &\Rightarrow& \beta=\alpha\quad \text{or}\quad \alpha\beta=-1. 
\end{eqnarray*}
Since the columns of $B$ are orthogonal, $\alpha\neq \beta$. We show that $\alpha\beta=-1$ leads to a contradiction which in turn implies that no orthogonal $M\in\mathcal{S}(R_{7,2})$ exists. Hence, $q(R_{7,2})>2$. Suppose $\alpha\beta=-1$; then 
\[\begin{array}{rcl}
I_4 - \left [
\begin{array}{cc}
C_1^2&0 \\
0&C_2^2\end{array} \right ] &=& I_4 - C^2 \ = \
BB^T \ = \ \left[\begin{array}{rrr}
x_1 & x_2 & \alpha x_2\\
y_1 & y_2 & \beta y_2\end{array}\right] \left[\begin{array}{rr}
x_1^T & y_1^T\\
x_2^T & y_2^T\\
\alpha x_2^T & \beta y_2^T\end{array}\right]\\
\\
& =& \left[\begin{array}{rrr}
x_1x_1^T+(1+\alpha^2) x_2x_2^T   & 
x_1y_1^T+(1+\alpha\beta) x_2y_2^T\\
y_1x_1^T+(1+\alpha\beta) y_2x_2^T & y_1y_1^T+(1+\beta^2) y_2y_2^T
\end{array}\right] \\
\\
& =& \left[\begin{array}{rrr}
x_1x_1^T+(1+\alpha^2) x_2x_2^T   & 
x_1y_1^T\\
y_1x_1^T & y_1y_1^T+(1+\beta^2) y_2y_2^T
\end{array}\right].
\end{array}
\]
So $x_1y_1^T=0$. This implies that either $x_1=0$ or $y_1=0$ which is a contradiction since the entries of $B$ are nonzero. 

To show $q(R_{7,2})=3$, we see that $R_{7,2}$ results from joined duplication of a vertex of $G189$ in \cite{Butler} or \cite{MR3904092} (i.e., the graph obtained from $R_{7,2}$ by contracting edge $\{3,4\}$ is isomorphic to $G189$). From Table 3 in \cite{MR3904092} we find that $q(G189)=3$. Thus $q(R_{7,2})=3$.
\\

\noindent $R_{8,1}$: The graph $R_{8,1}$ is a closed candle, $R_{8,1}\cong H_4$. Thus $q(R_{8,1})=2$ by Lemma \ref{qCandle}.
\\

\noindent $R_{8,2}$: The following matrix $M_{8,2}$ is a matrix in $\mathcal{S}(R_{8,2})$,
\[
M_{8,2} = \frac{1}{\sqrt{5}}\left[\begin{array}{rrrrrrrr}
1 & 1 & 1 & 0 & 0 & 1 & 0 & -1 \\
1 & 1 & -1 & 0 & -1 & 0 & 0 & 1 \\
1 & -1 & 0 & \beta & 0 & 0 & \alpha & 0 \\
 0 & 0 & \beta & 0 & 1 & 1 & 0 & \alpha\\
0 &  -1 & 0 &  1 & -1 &  1 &  -1 & 0 \\
1 & 0 & 0 & 1 & 1 & -1 & -1 & 0 \\
0 & 0 & \alpha & 0 & -1 & -1 & 0 & \beta \\
-1 & 1 & 0 & \alpha & 0 & 0 & \beta & 0
\end{array}\right],
\]
where $\alpha=(\sqrt{5}+1)/2$ and $\beta=(\sqrt{5}-1)/2$. Since $M_{8,2}$ is orthogonal, $q(R_{8,2})=2$.
\\

\noindent $R_{8,3}$: The following matrix $M_{8,3}$ is a matrix in $\mathcal{S}(R_{8,3})$,
\[
M_{8,3} =\frac{1}{\sqrt{12}}\left[\begin{array}{rrrrrrrr}
2 & \sqrt{2} & 0 & -\sqrt{2} & \sqrt{2} & 0 & 0 & \sqrt{2} \\
\sqrt{2} & 0 & \sqrt{2} & 0 & -2 & -2 & 0 & 0 \\
0 & \sqrt{2} & -2 & \sqrt{2} & 0 & -\sqrt{2} & -\sqrt{2} & 0 \\
-\sqrt{2} & 0 & \sqrt{2} & 0 & 0 & 0 & -2 & 2 \\
 \sqrt{2} & -2 & 0 & 0 & 1 & 0 & -2 & -1\\
0 & -2 & -\sqrt{2} & 0 & 0 & -1 & 1 & 2 \\
0 & 0 & -\sqrt{2} & -2 & -2 & 1 & -1 & 0\\
\sqrt{2} & 0 & 0 & 2 & -1 & 2 & 0 & 1 
\end{array}\right].
\]
The matrix $M_{8,3}$ is orthogonal, so $q(R_{8,3})=2$.
\\

\noindent $R_{8,4}$: Note that in our labelling of $R_{8,4}$, $N_1(7)=\{1,3,6,8\}$ and $N_1(8)=\{1,3,6,7\}$. Let $G$ be the graph obtained from $R_{8,4}$ by contracting the edge $\{7,8\}$ to the vertex $(78)$ (and replacing every pair of multiple edges by a single edge). The following matrix $M$ is a matrix in $\mathcal{S}(G)$,
\[
M = \frac{1}{3}\left[\begin{array}{rrrrrrr}
0 & -\sqrt{3} & 0 & 0 & \sqrt{3} & 0 & \sqrt{3}\\
-\sqrt{3} & 1 & -\sqrt{3} & 1 & 1 & 0 & 0\\
0 & -\sqrt{3} & 0 & \sqrt{3} & 0 & 0 & -\sqrt{3}\\
0 & 1 & \sqrt{3} & 1 & 1 & \sqrt{3} & 0\\
\sqrt{3} & 1 & 0 & 1 & 1 & -\sqrt{3} & 0\\
0 & 0 & 0 & \sqrt{3} & -\sqrt{3} & 0 & \sqrt{3}\\
\sqrt{3} & 0 & -\sqrt{3} & 0 & 0 & \sqrt{3} & 0
\end{array}\right]
\]
(here the vertices are ordered as $\{1,2,3,4,5,6,(78)\}$). The matrix $M$ is orthogonal, so $q(G)=2$. Since $R_{8,4}$ is obtained from $G$ by joined duplication of $(78)$, $q(R_{8,4})=2$.
\\

\noindent $R_{8,5}$: Note that $R_{8,5}\cong K_4\Box K_2$. Thus by Corollary 6.8 in \cite{MR3118943}, we have $q(R_{8,5})=2$.
\\

\noindent $R_{8,6}$: The following matrix $M_{8,6}$ is a matrix in $\mathcal{S}(R_{8,6})$,
\[
M_{8,6}=\frac{1}{\sqrt{10}}\left[\begin{array}{rrrrrrrr}
-\sqrt{2} & \sqrt{2} & -1 & 0 & 0 & 0 & -1 & -2 \\
\sqrt{2} & \sqrt{2} & -2 & 1 & 0 & 0 & 0 & 1 \\
-1 & -2 & -\sqrt{2} & \sqrt{2} & 1 & 0 & 0 & 0 \\
0 & 1 & \sqrt{2} & \sqrt{2} & 2 & -1 & 0 & 0 \\
0 & 0 & 1 & 2 & -\sqrt{2} & \sqrt{2} & -1 & 0 \\
0 & 0 & 0 & -1 & \sqrt{2} & \sqrt{2} & -2 & 1 \\
-1 & 0 & 0 & 0 & -1 & -2 & -\sqrt{2} & \sqrt{2} \\
-2 & 1 & 0 & 0 & 0 & 1 & \sqrt{2} & \sqrt{2}
\end{array}\right].
\]
The matrix $M_{8,6}$ is orthogonal, so $q(R_{8,6})=2$.
\\

\noindent $R_{9,1}$: Consider a matrix $M\in\mathcal{S}(R_{9,1})$ where we use variables for each edge and vertex. That is,
\[
[M]_{ij}=\begin{cases}
x_{ij} & \text{if $ij\in E(R_{9,1})$,}\\
x_{ii} & \text{if $i=j$, and}\\
0 & \text{if $ij\notin E(R_{9,1})$.}
\end{cases}
\]

Suppose $M$ is an orthogonal matrix. Note that the edges of $R_{9,1}$ can be partitioned into the $9$-cycle $(1,2,3,4,5,6,7,8,9,1)$, and the $3$-cycles $(1,4,7,1)$, $(2,5,8,2)$, and $(3,6,9,3)$. Since the edge $\{1,2\}$ does not lie in any triangle in $R_{9,1}$, we have
\[
0 = [M^2]_{12}=x_{11}x_{12}+x_{12}x_{22}.
\]
Since $x_{12}\neq 0$, we conclude that $x_{22}=-x_{11}$. Repeating this argument for each edge of the $9$-cycle, we see that $x_{ii}=-x_{ii}$, or $x_{ii}=0$ for all $1\leq i\leq 9$.

Now consider the $3$-cycle $(1,4,7,1)$. Taking account of the walks of length $2$ between $1$ and $7$, we have
\[
0 = [M^2]_{17}=x_{11}x_{17}+x_{17}x_{77}+x_{14}x_{47}=x_{14}x_{47}.
\]
But since $x_{14},x_{47}\neq 0$, this is impossible. Thus there is no orthogonal matrix $M\in\mathcal{S}(R_{9,1})$, and we conclude $q(R_{9,1})>2$.
\\

\noindent $R_{9,2}$: Consider a matrix $M\in\mathcal{S}(R_{9,2})$ where we use variables for each edge and vertex. That is,
\[
[M]_{ij}=\begin{cases}
x_{ij} & \text{if $ij\in E(R_{9,2})$,}\\
x_{ii} & \text{if $i=j$, and}\\
0 & \text{if $ij\notin E(R_{9,2})$.}
\end{cases}
\]

Suppose $M$ is an orthogonal matrix. Note that $\{6,7\}\in E(R_{9,2})$, but $\{6,7\}$ is not included in any $3$-cycle in $R_{9,2}$. Thus
\[
0 = [M^2]_{67}=x_{66}x_{67}+x_{67}x_{77}.
\]
Since the variable $x_{67}\neq 0$, this implies that $x_{66}=-x_{77}$. Similarly, we see that edges $\{2,6\}$ and $\{3,7\}$ are not included in any $3$-cycles in $R_{9,2}$. Considering $[M^2]_{26}$ and $[M^2]_{37}$ we derive $x_{22}=-x_{66}$ and $x_{33}=-x_{77}$. Combining these three equations, we have $x_{22}=-x_{33}$. Now consider the paths of length $2$ between vertices $2$ and $3$. We have
\[
0 = [M^2]_{23}=x_{22}x_{23}+x_{23}x_{33}+x_{12}x_{13}=x_{12}x_{13}.
\]
But since $x_{12},x_{13}\neq 0$, this is impossible. Thus there is no orthogonal matrix $M\in\mathcal{S}(R_{9,2})$, and we conclude $q(R_{9,2})>2$.
\\

\noindent $R_{9,3}$: Note that $\{1,2,3\}$, $\{4,5,9\}$, and $\{6,7,8\}$ all induce $K_3$ subgraphs in $R_{9,3}$. Moreover, $\{1,8,9\}$, $\{2,5,6\}$, and $\{3,4,7\}$ also induce $K_3$ subgraphs in $R_{9,3}$, and we see that $R_{9,3}\cong K_3\Box K_3$. We prove in Lemma \ref{qKmKn} that $q(K_m\Box K_n)=3$ for all $m,n\geq 3$. This establishes $q(R_{9,3})=3$.
\\

\noindent $R_{10,1}$: The graph $R_{10,1}$ is a closed candle, $R_{10,1}\cong H_5$. Thus $q(R_{10,1})=2$ by Lemma \ref{qCandle}. 
\end{proof}

In order to complete the preceding proof, we establish the following lemma, which shows that the bound in Proposition 3.1 of \cite{MR3904092} is sharp for complete graphs.
\begin{lemma}\label{qKmKn}
For $m,n\geq 3$, we have $q(K_m\Box K_n)=3.$
\end{lemma}
\begin{proof}
Since $q(K_s)=2$ for any $s\geq 2$ it follows from a basic application of Kronecker products that $q(K_m\Box K_n)\leq 3$ (this inequality can also be deduced from Proposition 3.1 in \cite{MR3904092}).
It remains to verify that in fact $q(K_m\Box K_n)\geq 3$.

If $q(K_m\Box K_n)=2,$ then there exists a matrix $C\in \mathcal{S}(K_m\Box K_n)$ that satisfies $C^2=I$ and is given by
\[
C=\left[\begin{array}{ccccc}
A_{11} & D_{12} & D_{13} & \ldots & D_{1m}\\
D_{12} & A_{22} & D_{23} & \ldots & D_{2m}\\
&& \ddots\\
D_{1m} & D_{2m} & D_{3m} & \ldots & A_{mm}\\
\end{array}\right],
\]
where $D_{ij}$ is an $n\times n$ diagonal matrix with nonzero diagonal entries for each $1\leq i,j\leq m$ and $A_{ii}\in \mathcal{S}(K_n)$ for $1\leq i\leq m$.

Let $[C^2]_{ij}$ denote the $(i,j)$ block of the matrix $C^2$ partitioned conformally with $C$ above. Then
\begin{align*}
[C^2]_{12}&=A_{11}D_{12} +D_{12}A_{22}+\sum_{j=3}^{m}{D_{1j}}D_{{2j}}=0,\\
[C^2]_{13}&=A_{11}D_{13} +D_{13}A_{33}+\sum_{j\neq 1, 3}{D_{1j}}D_{{3j}}=0,\\
[C^2]_{23}&=A_{22}D_{23} +D_{23}A_{33}+\sum_{j\neq 2,3}{D_{2j}}D_{{3j}}=0.
\end{align*}

Since $D_{ij}$ is invertible, we have 

\begin{align}
  \label{eq0}  A_{22} &=-D_{12}^{-1}A_{11}D_{12}-D_{12}^{-1}\left(\sum_{j=3}^{m}{D_{1j}}D_{{2j}}\right),\\ \label{eq1}
    A_{33}&= -D_{13}^{-1}A_{11}D_{13}-D_{13}^{-1}\left(\sum_{j\neq 1, 3}{D_{1j}}D_{{3j}}\right),\\\label{eq2}
    A_{22}&=-D_{23}A_{33}D_{23}^{-1} -\left(\sum_{j\neq 2,3}{D_{2j}}D_{{3j}}\right)D_{23}^{-1}.
\end{align}

From \eqref{eq1} and \eqref{eq2} we have

 \begin{equation}
 \label{eq3}
A_{22}=D_{23}D_{13}^{-1}A_{11}D_{13}D_{23}^{-1}+
D_{23}D_{13}^{-1}\left(\sum_{j\neq 1, 3}{D_{1j}}D_{{3j}}\right)D_{23}^{-1} -\left(\sum_{j\neq 2,3}{D_{2j}}D_{{3j}}\right)D_{23}^{-1}.
\end{equation}
Note that $[D_{12}]_{ij}=c_{i(n+j)}$, $[D_{13}]_{ij}=c_{i(2n+j)},$ and $[D_{23}]_{ij}=c_{(n+i)(2n+j)}$, and, by assumption, each such entry is nonzero when $i=j$. By direct calculation we have
\[
[D_{23}D_{13}^{-1}A_{11}D_{13}D_{23}^{-1}]_{ij}=\displaystyle \frac{c_{ij}c_{(n+i)(2n+i)}c_{j(2n+j)}}{c_{i(2n+i)}c_{(n+j)(2n+j)}}
\]
and
\[
[D_{12}^{-1}A_{11}D_{12}]_{ij}=\displaystyle \frac{c_{ij}c_{j(n+j)}}{c_{i(n+i)}}.
\]

Using \eqref{eq0} and \eqref{eq3}, and letting $(i,j)=(1,2), (1,3),$ and $(2,3)$ (there are no contributions from the diagonal terms in (\ref{eq0}) nor (\ref{eq3}))  we have 

\begin{align}
 \label{eq4} \displaystyle \frac{c_{12}c_{(n+1)(2n+1)}c_{2(2n+2)}}{c_{1(2n+1)}c_{(n+2)(2n+2)}} & =  -\displaystyle  \frac{c_{12}c_{2(n+2)}}{c_{1(n+1)}} \\
\label{eq5} \displaystyle \frac{c_{13}c_{(n+1)(2n+1)}c_{3(2n+3)}}{c_{1(2n+1)}c_{(n+3)(2n+3)}} & =  -\displaystyle \frac{c_{13}c_{3(n+3)}}{c_{1(n+1)}}\\
 \label{eq6} \displaystyle \frac{c_{23}c_{(n+2)(2n+2)}c_{3(2n+3)}}{c_{2(2n+2)}c_{(n+3)(2n+3)}} & =  -\displaystyle \frac{c_{23}c_{3(n+3)}}{c_{2(n+2)}}.
\end{align}

Manipulating equations \eqref{eq4} and \eqref{eq5} produces the equation
%
%
%
%
%

\begin{eqnarray*}
 \displaystyle \frac{c_{2(2n+2)}c_{(n+3)(2n+3)}}{c_{(n+2)(2n+2)}c_{3(2n+3)}}
 & = & \displaystyle  \frac{c_{2(n+2)}}{c_{3(n+3)}}
\end{eqnarray*}
or
\begin{eqnarray*}
 c_{3(n+3)}c_{2(2n+2)}c_{(n+3)(2n+3)}
 & = & c_{(n+2)(2n+2)}c_{3(2n+3)}c_{2(n+2)}.
\end{eqnarray*}
However from \eqref{eq6} we have 
\begin{eqnarray*}
c_{(n+2)(2n+2)}c_{3(2n+3)}c_{2(n+2)}=-c_{2(2n+2)}c_{(n+3)(2n+3)}c_{3(n+3)}.
\end{eqnarray*}
which is a contradiction.
\end{proof}

This completes the verification of the $q$-values listed in Table \ref{RGraphTable}, which completes the proof of Theorem \ref{Regular4} for graphs with diameter at most $2$.

\section{Sporadic 4-Regular Graphs and Theorem \ref{Regular4}}\label{SporadicSection}

In this section we establish whether or not $q(G)=2$ for a collection of sporadic graphs. These graphs are all connected with diameter at least $3$ and arise within the proofs in Section \ref{LargeDiameter}.

\begin{lemma}\label{SporadicGraphs}
Table \ref{SGraphTable} lists candidate $4$-regular graphs with diameter at least 3 (also shown in Figure \ref{RGraphs2}) and includes their $q$-values (or a bound on their $q$-value). 

\begin{table}[!ht]
\begin{center}
\begin{tabular}{r|r}
Graph & $q$-value\\
\hline
$R_{10,2}$ & 2\\
$R_{10,3}$ & 2\\
$R_{10,4}$ & 2\\
$R_{12,1}$ & $>2$\\
\end{tabular}
\quad\quad\quad
\begin{tabular}{r|r}
Graph & $q$-value\\
\hline
$R_{12,2}$ & $>2$\\
$R_{12,3}$ & 2\\
$R_{14,1}$ & 2\\
 & \\
\end{tabular}
\caption{Sporadic $4$-Regular Graphs with diameter at least 3.}\label{SGraphTable}
\end{center}
\end{table}
\end{lemma}
\begin{figure}
\centering
\begin{tikzpicture}[scale=1.1]
	\node (a0)  [draw,circle,inner sep=1pt] at (-1.5,0.6) {1};
	\node (a1)  [draw,circle,inner sep=1pt] at (-1.5,2.4) {2};
	\node (a2)  [draw,circle,inner sep=1pt] at (-0.5,1) {3};
	\node (a3)  [draw,circle,inner sep=1pt] at (0.5,2) {4};
	\node (a4)  [draw,circle,inner sep=1pt] at (1.5,0.6) {5};
	\node (a5)  [draw,circle,inner sep=1pt] at (1.5,2.4)  {6};
	\node (a6)  [draw,circle,inner sep=1pt] at (0.5,1)  {7};
	\node (a7)  [draw,circle,inner sep=1pt] at (-0.5,2)  {8};
	\node (a8)  [draw,circle,inner sep=1pt] at (0,3)  {9};
	\node (a9)  [draw,circle,inner sep=1pt] at (0,0)  {10};
	\node (label1) at (0,-0.75) {$R_{10,2}$};
	\draw (a0) edge (a1);
	\draw (a0) edge (a4);
	\draw (a0) edge (a7);
	\draw (a0) edge (a9);
	\draw (a1) edge (a2);
	\draw (a1) edge (a5);
	\draw (a1) edge (a8);
	\draw (a2) edge (a3);
	\draw (a2) edge (a6);
	\draw (a2) edge (a9);
	\draw (a3) edge (a4);
	\draw (a3) edge (a7);
	\draw (a3) edge (a8);
	\draw (a4) edge (a5);
	\draw (a4) edge (a9);
	\draw (a5) edge (a6);
	\draw (a5) edge (a8);
	\draw (a6) edge (a7);
	\draw (a6) edge (a9);
	\draw (a7) edge (a8);
    \node (b1)  [draw,circle,inner sep=1pt] at (-1.5+4,2) {1};
	\node (b2)  [draw,circle,inner sep=1pt] at (-0.5+4,2) {2};
	\node (b3)  [draw,circle,inner sep=1pt] at (0+4,3) {3};
	\node (b4)  [draw,circle,inner sep=1pt] at (-1.5+4,1) {4};
	\node (b5)  [draw,circle,inner sep=1pt] at (-0.5+4,1) {5};
	\node (b6)  [draw,circle,inner sep=1pt] at (0.5+4,1)  {6};
	\node (b7)  [draw,circle,inner sep=1pt] at (1.5+4,1)  {7};
	\node (b8)  [draw,circle,inner sep=1pt] at (0.5+4,2)  {8};
	\node (b9)  [draw,circle,inner sep=1pt] at (0+4,0)  {9};
	\node (b10)  [draw,circle,inner sep=1pt] at (1.5+4,2)  {10};
    \node (label2) at (0+4,-0.75) {$R_{10,3}$};
	\draw (b1) edge (b2);
	\draw (b1) edge (b3);
	\draw (b1) edge (b4);
	\draw (b1) edge (b5);
	\draw (b2) edge (b3);
	\draw (b2) edge (b4);
	\draw (b2) edge (b5);
	\draw (b6) edge (b7);
	\draw (b6) edge (b8);
	\draw (b6) edge (b9);
	\draw (b6) edge (b10);
	\draw (b7) edge (b8);
	\draw (b7) edge (b9);
	\draw (b7) edge (b10);
	\draw (b3) edge (b8);
	\draw (b3) edge (b10);
	\draw (b4) edge (b9);
	\draw (b5) edge (b9);
	\draw (b4) edge (b8);
	\draw (b5) edge (b10);
    \node (1)  [draw,circle,inner sep=1pt] at (1*1.5+8,0+1.5) {1};
	\node (2)  [draw,circle,inner sep=1pt] at (0.809*1.5+8,0.588*1.5+1.5) {2};
	\node (3)  [draw,circle,inner sep=1pt] at (0.309*1.5+8,0.951*1.5+1.5) {3};
	\node (4)  [draw,circle,inner sep=1pt] at (-0.309*1.5+8,0.951*1.5+1.5) {4};
	\node (5)  [draw,circle,inner sep=1pt] at (-0.809*1.5+8,0.588*1.5+1.5) {5};
	\node (6)  [draw,circle,inner sep=1pt] at (-1*1.5+8,0+1.5)  {6};
	\node (7)  [draw,circle,inner sep=1pt] at (-0.809*1.5+8,-0.588*1.5+1.5)  {7};
	\node (8)  [draw,circle,inner sep=1pt] at (-0.309*1.5+8,-0.951*1.5+1.5)  {8};
	\node (9)  [draw,circle,inner sep=1pt] at (0.309*1.5+8,-0.951*1.5+1.5)  {9};
	\node (10)  [draw,circle,inner sep=1pt] at (0.809*1.5+8,-0.588*1.5+1.5)  {10};
    \node (label7) at (0+8,-0.75) {$R_{10,4}$};
	\draw (1) edge (2);
	\draw (2) edge (3);
	\draw (3) edge (4);
	\draw (4) edge (5);
	\draw (5) edge (6);
	\draw (7) edge (6);
	\draw (7) edge (8);
	\draw (9) edge (8);
	\draw (9) edge (10);
	\draw (1) edge (10);
	\draw (1) edge (4);
	\draw (7) edge (4);
	\draw (7) edge (10);
	\draw (3) edge (10);
	\draw (3) edge (6);
	\draw (9) edge (6);
	\draw (9) edge (2);
	\draw (5) edge (2);
	\draw (5) edge (8);
	\draw (1) edge (8);
    \node (c0)  [draw,circle,inner sep=1pt] at (-1.5,1-5) {1};
	\node (c1)  [draw,circle,inner sep=1pt] at (-1.5,2-5) {2};
	\node (c2)  [draw,circle,inner sep=1pt] at (-0.5,1-5) {3};
	\node (c3)  [draw,circle,inner sep=1pt] at (-0.5,2-5) {4};
	\node (c4)  [draw,circle,inner sep=1pt] at (0.5,1-5) {5};
	\node (c5)  [draw,circle,inner sep=1pt] at (0.5,2-5)  {6};
	\node (c6)  [draw,circle,inner sep=1pt] at (1.5,1-5)  {7};
	\node (c7)  [draw,circle,inner sep=1pt] at (1.5,2-5)  {8};
	\node (c8)  [draw,circle,inner sep=1pt] at (-0.5,0-5)  {9};
	\node (c9)  [draw,circle,inner sep=1pt] at (-0.5,3-5)  {10};
	\node (c10)  [draw,circle,inner sep=1pt] at (0.5,0-5)  {11};
	\node (c11)  [draw,circle,inner sep=1pt] at (0.5,3-5)  {12};
    \node (label3) at (0,-0.75-5) {$R_{12,1}$};
	\draw (c0) edge (c1);
	\draw (c2) edge (c1);
	\draw (c2) edge (c3);
	\draw (c4) edge (c3);
	\draw (c4) edge (c5);
	\draw (c6) edge (c5);
	\draw (c6) edge (c7);
	\draw (c0) edge (c7);
	\draw (c8) edge (c0);
	\draw (c8) edge (c2);
	\draw (c8) edge (c4);
	\draw (c8) edge (c6);
	\draw (c10) edge (c0);
	\draw (c10) edge (c2);
	\draw (c10) edge (c4);
	\draw (c10) edge (c6);
	\draw (c9) edge (c1);
	\draw (c9) edge (c3);
	\draw (c9) edge (c5);
	\draw (c9) edge (c7);
	\draw (c11) edge (c1);
	\draw (c11) edge (c3);
	\draw (c11) edge (c5);
	\draw (c11) edge (c7);
    \node (d1)  [draw,circle,inner sep=1pt] at (0+4,0-5) {1};
	\node (d2)  [draw,circle,inner sep=1pt] at (-1.5+4,1-5) {2};
	\node (d3)  [draw,circle,inner sep=1pt] at (-0.5+4,1-5) {3};
	\node (d4)  [draw,circle,inner sep=1pt] at (0.5+4,1-5) {4};
	\node (d5)  [draw,circle,inner sep=1pt] at (1.5+4,1-5) {5};
	\node (d6)  [draw,circle,inner sep=1pt] at (-1.5+4,2-5)  {6};
	\node (d7)  [draw,circle,inner sep=1pt] at (-0.75+4,2-5)  {7};
	\node (d8)  [draw,circle,inner sep=1pt] at (0+4,2-5)  {8};
	\node (d9)  [draw,circle,inner sep=1pt] at (0.75+4,2-5)  {9};
	\node (d10)  [draw,circle,inner sep=1pt] at (1.5+4,2-5)  {10};
	\node (d11)  [draw,circle,inner sep=1pt] at (-0.5+4,3-5)  {11};
	\node (d12)  [draw,circle,inner sep=1pt] at (0.5+4,3-5)  {12};
    \node (label4) at (0+4,-0.75-5) {$R_{12,2}$};
	\draw (d1) edge (d2);
	\draw (d1) edge (d3);
	\draw (d1) edge (d4);
	\draw (d1) edge (d5);
	\draw (d6) edge (d2);
	\draw (d6) edge (d3);
	\draw (d6) edge (d4);
	\draw (d6) edge (d11);
	\draw (d7) edge (d2);
	\draw (d7) edge (d5);
	\draw (d7) edge (d11);
	\draw (d7) edge (d12);
	\draw (d8) edge (d2);
	\draw (d8) edge (d5);
	\draw (d8) edge (d11);
	\draw (d8) edge (d12);
	\draw (d9) edge (d3);
	\draw (d9) edge (d4);
	\draw (d9) edge (d11);
	\draw (d9) edge (d12);
	\draw (d10) edge (d3);
	\draw (d10) edge (d4);
	\draw (d10) edge (d5);
	\draw (d10) edge (d12);
    \node (e1)  [draw,circle,inner sep=1pt] at (0+8,0-5) {1};
	\node (e2)  [draw,circle,inner sep=1pt] at (-1.5+8,1-5) {7};
	\node (e3)  [draw,circle,inner sep=1pt] at (-0.5+8,1-5) {8};
	\node (e4)  [draw,circle,inner sep=1pt] at (0.5+8,1-5) {9};
	\node (e5)  [draw,circle,inner sep=1pt] at (1.5+8,1-5) {10};
	\node (e6)  [draw,circle,inner sep=1pt] at (-1.5+8,2-5)  {2};
	\node (e7)  [draw,circle,inner sep=1pt] at (-0.75+8,2-5)  {3};
	\node (e8)  [draw,circle,inner sep=1pt] at (0+8,2-5)  {4};
	\node (e9)  [draw,circle,inner sep=1pt] at (0.75+8,2-5)  {5};
	\node (e10)  [draw,circle,inner sep=1pt] at (1.5+8,2-5)  {6};
	\node (e11)  [draw,circle,inner sep=1pt] at (-0.5+8,3-5)  {11};
	\node (e12)  [draw,circle,inner sep=1pt] at (0.5+8,3-5)  {12};
    \node (label5) at (0+8,-0.75-5) {$R_{12,3}$};
	\draw (e1) edge (e2);
	\draw (e1) edge (e3);
	\draw (e1) edge (e4);
	\draw (e1) edge (e5);
	\draw (e6) edge (e2);
	\draw (e6) edge (e3);
	\draw (e6) edge (e4);
	\draw (e6) edge (e11);
	\draw (e7) edge (e2);
	\draw (e7) edge (e5);
	\draw (e7) edge (e11);
	\draw (e7) edge (e12);
	\draw (e8) edge (e2);
	\draw (e8) edge (e4);
	\draw (e8) edge (e11);
	\draw (e8) edge (e12);
	\draw (e9) edge (e3);
	\draw (e9) edge (e5);
	\draw (e9) edge (e11);
	\draw (e9) edge (e12);
	\draw (e10) edge (e3);
	\draw (e10) edge (e4);
	\draw (e10) edge (e5);
	\draw (e10) edge (e12);
    \node (f1)  [draw,circle,inner sep=1pt] at (2+4,0-8.5) {1};
	\node (f2)  [draw,circle,inner sep=1pt] at (1.246+4,1.564-8.5) {2};
	\node (f3)  [draw,circle,inner sep=1pt] at (-0.446+4,1.95-8.5) {3};
	\node (f4)  [draw,circle,inner sep=1pt] at (-1.802+4,0.868-8.5) {4};
	\node (f5)  [draw,circle,inner sep=1pt] at (-1.802+4,-0.868-8.5) {5};
	\node (f6)  [draw,circle,inner sep=1pt] at (-0.446+4,-1.95-8.5)  {6};
	\node (f7)  [draw,circle,inner sep=1pt] at (1.246+4,-1.564-8.5)  {7};
	\node (f8)  [draw,circle,inner sep=1pt] at (0.912+4,0.868-8.5)  {8};
	\node (f9)  [draw,circle,inner sep=1pt] at (-0.11+4,1.254-8.5)  {9};
	\node (f10)  [draw,circle,inner sep=1pt] at (-1.048+4,0.696-8.5)  {10};
	\node (f11)  [draw,circle,inner sep=1pt] at (-1.198+4,-0.386-8.5)  {11};
	\node (f12)  [draw,circle,inner sep=1pt] at (-0.446+4,-1.178-8.5)  {12};
	\node (f13)  [draw,circle,inner sep=1pt] at (0.642+4,-1.082-8.5)  {13};
	\node (f14)  [draw,circle,inner sep=1pt] at (1.246+4,-0.172-8.5)  {14};
	\node (label6) at (0+4,-0.75-10.5) {$R_{14,1}$};
	\draw (f1) edge (f8);
	\draw (f1) edge (f10);
	\draw (f1) edge (f13);
	\draw (f1) edge (f14);
	\draw (f2) edge (f8);
	\draw (f2) edge (f9);
	\draw (f2) edge (f11);
	\draw (f2) edge (f14);
	\draw (f3) edge (f8);
	\draw (f3) edge (f9);
	\draw (f3) edge (f10);
	\draw (f3) edge (f12);
	\draw (f4) edge (f9);
	\draw (f4) edge (f10);
	\draw (f4) edge (f11);
	\draw (f4) edge (f13);
	\draw (f5) edge (f10);
	\draw (f5) edge (f11);
	\draw (f5) edge (f12);
	\draw (f5) edge (f14);
	\draw (f6) edge (f8);
	\draw (f6) edge (f11);
	\draw (f6) edge (f12);
	\draw (f6) edge (f13);
	\draw (f7) edge (f9);
	\draw (f7) edge (f12);
	\draw (f7) edge (f13);
	\draw (f7) edge (f14);
\end{tikzpicture}
\caption{Sporadic $4$-regular graphs with diameter at least 3.}\label{RGraphs2}
\end{figure}

\begin{proof}


\noindent $R_{10,2}$: To see that $q(R_{10,2})=2$, note that the following matrix $M_{10,2}\in R_{10,2}$ is orthogonal,
\[
M_{10,2} = \frac{1}{4}\left[\begin{array}{rrrrrrrrrr}
-\sqrt{2} & -2 & 0 & 0 & \sqrt{2} & 0 & 0 & 2 & 0 & -2 \\
-2 & \sqrt{2} & 2 & 0 & 0 & \sqrt{2} & 0 & 0 & -2 & 0 \\
0 & 2 & -\sqrt{2} & -2 & 0 & 0 & \sqrt{2} & 0 & 0 & -2 \\
0 & 0 & -2 & \sqrt{2} & 2 & 0 & 0 & -\sqrt{2} & -2 & 0 \\
\sqrt{2} & 0 & 0 & 2 & -\sqrt{2} & 2 & 0 & 0 & 0 & -2 \\
0 & \sqrt{2} & 0 & 0 & 2 & \sqrt{2} & -2 & 0 & 2 & 0 \\
0 & 0 & \sqrt{2} & 0 & 0 & -2 & -\sqrt{2} & -2 & 0 & -2 \\
2 & 0 & 0 & -\sqrt{2} & 0 & 0 & -2 & \sqrt{2} & -2 & 0 \\
0 & -2 & 0 & -2 & 0 & 2 & 0 & -2 & 0 & 0 \\
-2 & 0 & -2 & 0 & -2 & 0 & -2 & 0 & 0 & 0
\end{array}\right].
\]
\\

\noindent $R_{10,3}$: Note that contracting edges $\{1,2\}$ and $\{6,7\}$ in $R_{10,3}$ gives a graph isomorphic to $Q_3$. Since $R_{10,3}$ can be obtained from $Q_3$ by joined duplicating a pair of antipodal vertices, Corollary 3.3 from \cite{Butler} implies that $q(R_{10,3})\leq q(Q_3)=2$. Thus $q(R_{10,3})=2$.
\\

%

\noindent $R_{10,4}$: Theorem 5.3 in \cite{MR3999789} shows that $K_{n,n}$ with a perfect matching deleted has $q$-value 2 for all $n\neq 1,3$. Since $R_{10,4}$ is isomorphic to $K_{5,5}$ with a perfect matching deleted, $q(R_{10,4})=2$. 
\\

\noindent $R_{12,1}$: Following the example after Lemma \ref{OrthLemma}, set $V_1=\{1,3,5,7,10,12\}$ and $V_2=\{2,4,6,8,9,11\}$. Assume that there exists an $M \in S(R_{12,1})$ with $M^2=I$ and $M$ is of the form
\[
M=\left[\begin{array}{cc}
C & B\\
B^T & D
\end{array}\right],
\] where $C$ and $D$ are diagonal matrices. Then by Lemma \ref{OrthLemma} we know that $B$ must be an orthogonal matrix. Consider the last two rows of $B$ (whose nonzero patterns are the same since vertices 10 and 12 have the same neighbors) and suppose they are equal to $u=[a,b,c,d,0,0]$ and $v=[x,y,z,w,0,0]$. Since both $u$ and $v$ are orthogonal to row 4 of $B$ we may deduce that 
$[z,w]=\alpha [c,d]$ for some nonzero scalar $\alpha$. Similarly since both $u$ and $v$ are orthogonal to row 2 and to row 3 we conclude that $[y,z]=\beta [b,c]$ and $[x,y]=\gamma [a,b]$, where $\beta$ and $\gamma$ are nonzero. Now it follows that $\alpha=\beta=\gamma$ and hence row 5 is a multiple of row 6 which contradicts the assumption that $B$ is orthogonal. Hence $q(R_{12,1})>2$.
\\

\noindent $R_{12,2}$:
Using a similar set up as in the case of $R_{12,1}$ set $V_1= \{2,3,4,5,11,12\}$ and $V_2=\{1,6,7,8,9,10\}$. Assume that there exists an $M \in S(R_{12,1})$ with $M^2=I$ and $M$ is of the form
\[
M=\left[\begin{array}{cc}
C & B\\
B^T & D
\end{array}\right],
\] where $C$ and $D$ are diagonal matrices. Then by Lemma \ref{OrthLemma} we know that $B$ must be an orthogonal matrix. Consider rows 2 and 3 of $B$ (whose nonzero patterns are the same since vertices 3 and 4 have the same neighbors). Following a similar argument as used for the graph $R_{12,1}$ (both rows 2 and 3 must be orthogonal to rows 1,5, and 6) we can deduce that rows 2 and 3 must be multiples of one another which is a contradiction. Hence $q(R_{12,2})>2$.
\\

\noindent $R_{12,3}$: To see that $q(R_{12,3})=2$, we present an orthogonal matrix $M_{12,3}\in \mathcal{S}(R_{12,3})$ in block form, as above. We take $V_1=\{1,2,3,4,5,6\}$ and $V_2=\{7,8,9,10,11,12\}$. Let
\[
B_{12,3} = \left[\begin{array}{rrrrrr}
-\zeta & -\beta & \zeta & -\sqrt{5}/5 & 0 & 0\\
 \varepsilon & -\gamma & \beta & 0 & \delta & 0\\
 \sqrt{5}/5 & 0 & 0 & -\zeta & -\beta & -\zeta\\
 \sqrt{5}/5 & 0 & \sqrt{5}/5 & 0 & -\varepsilon & \zeta\\
 0 & \delta & 0 & -\varepsilon & \gamma & \beta\\
 0 & -\varepsilon & -\zeta & -\sqrt{5}/5 & 0 & \sqrt{5}/5\\
\end{array}\right] ,
\]
where
\[
\alpha = (\sqrt{5} + 1)/2,\ \ \ \beta = \sqrt{2\sqrt{5}/5 - 4/5},\ \ \ 
\gamma = \sqrt{-7\sqrt{5}/10 + 17/10},\ \ \ \delta = \sqrt{5}/10 + 1/2,
\]
\[
 \varepsilon = \alpha\beta,\ \ \text{and}\ \ \zeta = \alpha\gamma.\\
\]
Thus
\[
M_{12,3} = \left[\begin{array}{cc}
0 & B_{12,3}\\
B_{12,3}^T & 0
\end{array}\right]
\]
has the desired properties.
\\

%











%
%
%

\noindent $R_{14,1}$: To see that $q(R_{14,1})=2$, we provide an orthogonal matrix $M_{14,1}\in \mathcal{S}(R_{14,1})$ in block form. We take $V_1=\{1,2,3,4,5,6,7\}$ and $V_2=\{8,9,10,11,12,13,14\}$, and let
\[
B_{14,1} = \frac{1}{2}\left[\begin{array}{rrrrrrr}
  1 & 0 & 1 & 0 & 0 & -1 & 1 \\
  1 & 1 & 0 & 1 & 0 & 0 & -1 \\
  -1 & 1 & 1 & 0 & 1 & 0 & 0 \\
  0 & -1 & 1 & 1 & 0 & 1 & 0 \\
  0 & 0 & -1 & 1 & 1 & 0 & 1 \\
  1 & 0 & 0 & -1 & 1 & 1 & 0 \\
  0 & 1 & 0 & 0 & -1 & 1 & 1 \\
\end{array}\right].
\]
Then
\[
M_{14,1}=\left[\begin{array}{cc}
0 & B_{14,1}\\
B_{14,1}^T & 0
\end{array}\right]
\]
has the desired properties. We note that $M_{14,1}$ was presented in \cite{MR2360149}.

%
%

\end{proof}

\section{Proof of Theorem \ref{Regular4} for 4-Regular Graphs with Large Diameter}\label{LargeDiameter}

In this section we complete the proof of items (3), (4) and (5) of Theorem \ref{Regular4} for graphs with diameter at least $3$. Recall from Section \ref{Preliminaries}, in the distance partition from any vertex $v$, every $u\in N_i(v)$ has at least two neighbors in $N_{i-1}(v)$ for all $i\geq 2$. These edges account for $2n-2-\deg(v)=2n-6$ of the edges of $G$. We consider the possible locations of the six extra edges not accounted for by these predecessors. Throughout the proofs in this section, once all of the edges incident with a vertex have been accounted for, we say that the vertex is \emph{full}.

\begin{lemma} \label{Max3LowEdges}  
Let $G$ be a connected $4$-regular graph with diameter at least $3$ and let $v$ be a vertex for which $\epsilon(v)\geq 3$.  Suppose that every vertex in $N_2(v)$ and $N_3(v)$ has exactly two predecessors and that $G[N_1(v)] \cup G[N_2(v)]$ contains at most three edges.  Then $q(G)>2$.
\end{lemma}

\begin{proof}
Assume $q(G) = 2$.\\

\noindent {\em Case 1:}  $N_1(v)$ is an independent set.   

\noindent Each pair of vertices of $N_1(v)$ share $v$ as a common neighbor, so each pair must share a common neighbor in $N_2(v)$ as well.  There are $\binom{4}{2} = 6$ such pairings, each resulting in a distinct vertex in $N_2(v)$ with two predecessors in $N_1(v)$ and accounting for all the edges between $N_1(v)$ and $N_2(v)$, also implying that $|N_2(v)| = 6$.

Suppose each vertex in $N_2(v)$ has a neighbor in $N_2(v)$.  Then there are three edges in $G[N_2(v)]$ and $G[N_2(v)]=3K_2$.  Let $xy$ be an edge in $N_2(v)$ and let $a,b$ be  the predecessors of $x$.  Then $a,b$ are not both predecessors of $y$.  Suppose that $b$ and $y$ are not adjacent.  Then $bxy$ is a unique shortest path of length 2 between $b$ and $y$.

Otherwise some vertex, say $u\in N_2(v)$, has no neighbor in $N_2(v)$.  Vertex $u$ shares a common neighbor in $N_1(v)$ with four other vertices in $N_2(v)$, call them $x_1, x_2,x_3, x_4$.  To avoid unique paths from $u$ to $x_i$ for $i=1,2,3,4$, $u$ and $x_i$ must have a common successor.  These common successors are distinct since each has just 2 predecessors. But then $\deg(u) \geq 6$.\\

\noindent {\em Case 2:}  $G[N_1(v)]$ contains exactly one edge.

\noindent This case follows the same logic as Case 1, with the exception that there are five pairings of nonadjacent vertices in $N_1(v)$ that share $v$ as a common neighbor so each of these pairs must have a common neighbor in $N_2(v)$.  As before, this accounts for all the edges between $N_1(v)$ and $N_2(v)$, so there are exactly five vertices in $N_2(v)$.  Since $G[N_2(v)]$ contains at most two edges, some vertex $u$ in $N_2(v)$ has no neighbors in $N_2(v)$.  Vertex $u$ shares a neighbor in $N_1(v)$ with at least three other vertices in $N_2(v)$ .  Arguing similarly to Case 1, $\deg(u) \geq 5$.\\

 \noindent {\em Case 3:}  $G[N_1(v)]$ contains exactly two edges.  

\noindent Suppose the two edges do not share a vertex.  Then the four pairs of nonadjacent vertices in $N_1(v)$ that have $v$ as a common neighbor must each have a common neighbor in $N_2(v)$, accounting for all the edges between $N_1(v)$ and $N_2(v)$ and resulting in exactly four vertices in $N_2(v)$. Since there is at most one edge in $N_2(v)$, there is a vertex $u$ in $N_2(v)$ that has no neighbor in $N_2(v)$.  Let $a$ be a neighbor of $u$ in $N_1(v)$ with $ab$ one of the two edges in $G[N_1(v)]$. Then $uab$ is a unique shortest path of length $2$ in $G$. 

Otherwise the two edges in $G[N_1(v)]$ share a vertex, say $c$. Then $c$ has exactly one neighbor $w$ in $N_2(v)$. Let the non-neighbor of $c$ in $N_1(v)$ be $d$. Since $c$ and $d$ have $v$ as a common neighbor, the other predecessor of $w$ in $N_1(v)$ must be $d$. To prevent a unique shortest path of length 2 between $w$ and the neighbors of $c$ in $N_1(v)$, there must be an edge connecting $w$ to a vertex, $y\in N_2(v)$ whose two predecessors are the two neighbors of $c$ in $N_1(v)$. But this creates a unique shortest path $dwy$ of length $2$. Note that in this case we did not use the hypothesis that each vertex in $N_3(v)$ has two predecessors.  This will be used later in the proof of Lemma \ref{No4pred}.\\

\noindent {\em Case 4:}  $G[N_1(v)]$ contains exactly three edges.

\noindent If $G[N_1(v)] \cong P_4$, then the two endpoints of the $P_4$ must have a common neighbor $w$ in $N_2(v)$. Then $w$ has a unique shortest path of length 2 to both non-end vertices of $P_4$.  

If $G[N_1(v)] \cong C_3 \cup K_1$, the same argument as in Case 3 applies, with $c$ being any vertex in the $C_3$.  Note again, that in these first  two possibilities that the hypothesis that each vertex in $N_3(v)$ has two predecessors was not used.  This will be used at the end of the proof of Lemma \ref{No4pred}.

Suppose $G[N_1(v)] \cong K_{1,3}$.  Then each pair of nonadjacent vertices in $N_1(v)$ has a common neighbor in $N_1(v)$, in addition to the common neighbor $v$.  But in order for each vertex in $N_1(v)$ to have degree 4, there must be three vertices in $N_2(v)$, each adjacent to a different pair of leaves of the $K_{1,3}$.  Each pair of vertices in $N_2(v)$ has one common neighbor in $N_1(v)$, and since $N_2(v)$ is an independent set, this implies that each pair of vertices in $N_2(v)$ must have a common neighbor in $N_3(v)$. These neighbors are distinct because each vertex in $N_3(v)$ has two predecessors. Note that all three vertices in $N_2(v)$ are full.  Let $a$ be a leaf vertex of the $K_{1,3}$; then there exist two paths of length 2, $asx$ and $aty$ for some $s,t \in N_2(v)$ and $x,y\in N_3(v)$, that are the unique $ax$ and $ay$ paths of length 2. 
\end{proof}

\begin{corollary} \label{No3or4pred}
Let $G$ be a connected $4$-regular graph with diameter at least $3$.  If $v$ is a vertex with $\epsilon(v)\geq 3$ and for which there is no vertex in the distance partition from $v$ with 3 or 4 predecessors, then $q(G)>2$.
\end{corollary}

\begin{proof}
Let $N_d(v)$ be the furthest distance set of $v$.  Since every vertex in $N_d(v)$ has exactly two predecessors, we know that $G[N_d(v)]$ is $2$-regular.   So at least three of the six extra edges must be in $N_d(v)$, and the maximum number of edges that could appear within the subgraphs  $G[N_i(v)]$ for $1\leq i \leq d-1$ is three. So the hypotheses of Lemma \ref{Max3LowEdges} are satisfied and $q(G) > 2$. 
\end{proof}

\begin{lemma} \label{No4pred}
Let $G$ be a connected $4$-regular graph with diameter at least $3$ and let $v$ be a vertex for which $\epsilon(v) \geq 3$. If in the distance partition from $v$ no vertex has four predecessors, then either $G\cong R_{10,3}$, $G\cong K_3\Box C_4$, or $q(G)>2$.
\end{lemma}

\begin{proof}
Suppose $q(G) = 2$.  We assume $G$ has a vertex with three predecessors as otherwise the result follows from Corollary \ref{No3or4pred}.  We begin by establishing three claims.\\

\noindent\emph{Claim \#1:} Any vertex with three predecessors cannot have a successor.

\noindent\emph{Proof of Claim \#1:}
Suppose some vertex $w$ in $N_i(v)$ has three predecessors in $N_{i-1}(v)$ and one successor $z$ in $N_{i+1}(v)$. Note that $i$ must be at least $2$.  Since $z$ has at most three predecessors in $N_i(v)$, there must be a neighbor $x$ of $z$ in $N_{i+1}(v)$ or $N_{i+2}(v)$. But then there is a unique shortest path of length 2 between $w$ and $x$, as $w$ is full.

Observe that the proof of Claim \#1 also implies that if $G$ is a graph and $v\in V(G)$ for which some vertex in the distance partition from $v$ has four predecessors, any vertex $w\in N_i(v)$ with three predecessors cannot have a successor $z$ unless $z$ is a vertex with four predecessors. We will use this in Case 4 of the proof of Lemma \ref{HardCase}.\\

\noindent\emph{Claim \#2:} Vertices with three predecessors occur in pairs, with each pair the endpoints of a path in some distance set $N_i(v)$.

\noindent\emph{Proof of Claim \#2:}
  Suppose $w_1 \in N_i(v)$ has three predecessors.  Since it can have no successors but has degree 4, it must have a neighbor $w_2\in N_i(v)$.  So $w_1$ is one end of a path $w_1 w_2 \ldots w_k$ in $N_i(v)$, where $w_2, \ldots, w_{k-1}$ have two predecessors each, and hence is full, and $w_{k-1}$ is the only neighbor of $w_k$ in $N_i(v)$.  If $w_k$ has a successor $y$, then $yw_kw_{k-1}$ is a unique shortest path of length 2 between $y$ and $w_{k-1}$.  So $w_k$ must have three predecessors.\\

\noindent\emph{Claim \#3:} A pair of vertices with three predecessors can occur only in $N_i(v)$ for $i\geq 3$.

\noindent\emph{Proof of Claim \#3:}  Let $w$ be a vertex with three predecessors. Suppose $w \in N_2(v)$.  Note that $w$ cannot have a successor. Let $z$ be another vertex in $N_2(v)$ with three predecessors, and no successor.  Then there are at least two vertices in $N_1(v)$ that are common predecessors of $w$ and $z$.  Since $\epsilon(v) \geq 3$, there must be a vertex $x \in N_2(v)$ that has a successor $y\in N_3(v)$.  So $x$ cannot have three predecessors and therefore must have two.  If a predecessor of $x$ is a common predecessor of $w$ and $z$, then there is a unique shortest path of length 2 between $y$ and this predecessor of $x$.  So $w$ and $z$ must have exactly two shared predecessors, and the two predecessors of $x$ each has exactly one of $\{w,z\}$ as a successor.  This implies that, via $N_1(v)$, there is a path of length 2 between $x$ and $w$ (and between $x$ and $z$).  Since these paths cannot be unique shortest paths, and $x$ cannot be adjacent to $w$ (as it would then be on the $wz$ path in $N_2(v)$, and hence have degree at least $5$), $x$ must have a common neighbor $t\neq z$ in $N_2(v)$ with $w$.  But then $yxt$ is a unique shortest path of length 2 between $y$ and $t$.\\

Let $S$ represent the set of vertices with three predecessors.  We have shown that no vertex in $S$ can have a successor and all vertices in $S$ must occur in pairs in the distance sets $N_i(v)$ for $i\geq 3$.  A pair of such vertices requires a minimum of three of the six additional edges.  It follows that $|S| = 2$ or $|S| = 4$, leaving at most three or zero edges that could appear in the subgraphs $G[N_i(v)]$ for $1\leq i \leq \epsilon(v)-1$, respectively.  Let $j$ represent the smallest value of $i$ for which $N_i(v)$ contains an element of $S$.  If $j>3$, then Lemma \ref{Max3LowEdges} implies that $q(G)>2$, so we only have to consider $j=3$.  \\

\noindent {\em Case 1:} {$N_1(v)$ is an independent set.}

\noindent By the argument in Case 1 of Lemma \ref{Max3LowEdges}, $N_2(v)$ consists of six vertices, such that each vertex shares one common neighbor in $N_1(v)$ with each of four other vertices in $N_2(v)$.  Consider $u\in N_2(v)$, and denote by $w$ the unique vertex in $N_2(v)$ with which $u$ does not share a neighbor in $N_1(v)$. \\

\noindent {\em Case 1(a):} $N_2(v)$  is also an independent set.

\noindent Let $y_1, y_2$ be the successors of $u$ in $N_3(v)$ and let $x_1, x_2, x_3, x_4$ be the vertices in $N_2(v)$ that share a unique common predecessor with $u$.  Since there cannot be a unique shortest path of length 2 between $u$ and $x_i$ for any $i$, each $x_i$ must be adjacent to $y_1$ or $y_2$.  Since neither $y_1$ nor $y_2$ has four predecessors, $y_1$ must be adjacent to two of the $x_i$ and $y_2$ to the other two. Then $y_1, y_2 \in S$.  Neither $y_1$ nor $y_2$ is adjacent to $w$, so repeating the same argument, there are successors $z_1, z_2 \in S$ of $w$ such that $z_1$ is adjacent to two of the $x_i$ and $z_2$ to the other two.  This accounts for all edges between $N_2(v)$ and $N_3(v)$, so $N_3(v)=\{ y_1, y_2, z_1, z_2 \} =S$.  Note that each $x_i$ is adjacent to exactly one of $y_1,y_2$ and one of $z_1, z_2$.  Since none of $y_1, y_2, z_1, z_2$ has a successor by Claim \#1, there are exactly two edges in $G[N_3(v)]$ that do not share an endpoint.  If $y_1 y_2$ is not an edge, then $y_1 u y_2$ is the unique shortest path of length 2.  So $y_1 y_2$ is an edge, and $z_1z_2$ is an edge.  Without loss of generality, suppose that $x_1$ is adjacent to $y_1$.  Then $x_1$ and $y_2$ are not adjacent and $x_1y_1y_2$ is a unique shortest path of length 2. \\

\noindent {\em Case 1(b):}  $N_2(v)$ is not independent, so $|S| = 2$.

\noindent  Suppose there is an edge $ut$ in $G[N_2(v)]$, not sharing an endpoint with any other edge in $G[N_2(v)]$.  Then there is a shortest path of length 2 through $N_1(v)$ from $u$ to at least three other vertices in $N_2(v)$. So for each of these three vertices there must be a shortest path of length 2 through $N_3(v)$.  Since $\deg(u)=4$, $u$ can only have one successor, so there is a single vertex $x \in N_3(v)$ on all three paths.  But then $x$ has four predecessors. 

 Now suppose $ut$ and $rt$ are edges in $G[N_2(v)]$.  Recall that $w$ is the vertex in $N_2(v)$ that shares no predecessors with $u$. Let $y$ and $z$ be the remaining vertices in $N_2(v)$. To avoid a unique shortest path of length 2 through $N_1(v)$ from $u$ to $y$ or $z$, there must be a vertex $x$ in $N_3(v)$ whose predecessors are $u$, $y$, and $z$.  Then $tux$ is a unique shortest path of length 2. \\

 Therefore Case 1 cannot occur, and $N_1(v)$ is never an independent set.\\

\noindent {\em Case 2:}  $G[N_1(v)]$ contains exactly one edge, so $|S| = 2$.

\noindent We follow the proof of Case 2 of Lemma \ref{Max3LowEdges} but make more careful note of the location of the edges in $N_2(v)$.  Let $a$ and $b$ represent the adjacent vertices in $N_1(v)$.  Let $u$ be the vertex in $N_2(v)$ whose two predecessors are $N_1(v)\setminus \{a,b\}$.  Each of the four vertices in $N_2(v)\setminus \{u\}$ has exactly one of $a$ or $b$ as a predecessor. Note that there are four shortest paths of length 2 between these vertices and the other vertex in $\{a,b\}$. Since these paths cannot be unique, there must be two independent edges in $N_2(v)$ that connect vertices with no common predecessor in $\{a,b\}$.  Note that if two adjacent vertices in $N_2(v)$ do not share a predecessor in $N_1(v)\setminus \{a,b\}$, then there are unique shortest paths of length 2 between the vertices $N_1(v)\setminus \{a,b\}$ and the neighbors of their successors in $N_2(v)$. So the independent edges between the vertices of $N_2(v)\setminus \{u\}$ connect vertices so that each pair has one common predecessor in $N_1(v)\setminus \{a,b\}$ and has one non-common predecessor in $\{a,b\}$.

Each endpoint of the independent edges in $N_2(v)$ has exactly one successor, and $u$ must have two successors. It follows that there are six edges between the vertices in $N_2(v)$ and the vertices in $N_3(v)$.  Since the two vertices in $S \cap N_3(v)$ each have three predecessors, there can be no other vertices in $N_3(v)$ and $S=N_3(v)$. From Claim \#1, these two vertices must be adjacent to each other. Each of the four vertices in $N_2(v)\setminus \{u\}$ are adjacent to one of the vertices in $S$. The two predecessors other than $u$ of each vertex in $S$ must be a pair of nonadjacent vertices in $N_2(v)\setminus \{u\}$ to avoid unique paths of length 2 between vertices in $N_2(v)\setminus \{u\}$ and their non-neighbor in $N_3(v)$. Additionally, the two successors of $a$ (respectively, $b$) must have a common successor in $S$, otherwise there is a unique shortest path of length 2 between $a$ (respectively, $b$) and a vertex in $S$.  This implies that the predecessors of a vertex in $S$ are $u$ and two nonadjacent vertices in $N_2(v)$ that share a predecessor in $\{a,b\}$.  Thus $G\cong K_3\Box C_4$.\\

\noindent {\em Case 3:}  $G[N_1(v)]$ contains exactly two edges, so $|S| = 2$.

\noindent Then $q(G) > 2$ by Lemma \ref{Max3LowEdges}, Case 3 and the observation at the end of its proof.\\

\noindent {\em Case 4:}  $G[N_1(v)]$ contains exactly three edges, so $|S| = 2$.

If $G[N_1(v)] \cong P_4$ or $G[N_1(v)] \cong C_3 \cup K_1$, then by Case 4 of Lemma \ref{Max3LowEdges} and the observations within its proof, there is a pair of nonadjacent vertices with a unique path of length 2 between them.   So we assume $G[N_1(v)] \cong K_{1,3}$.  As in that proof, $N_2(v)$ consists of three vertices, each being the common successor of two leaves of the $K_{1,3}$. Since $j=3$, there are two adjacent vertices of $S$ in $N_3(v)$, each with three predecessors.  Since $N_2(v)$ is an independent set, each of its  three  vertices has two successors, so there are six edges between $N_2(v)$ and $N_3(v)$. It can now be verified that $G\cong R_{10,3}$. For instance, let $v$ correspond to vertex $1$ in the drawing of $R_{10,3}$ in Figure \ref{RGraphs2}.\\

In all of the above cases, we either reach a contradiction to $q(G)=2$, or $G$ is isomorphic to either $R_{10,3}$ or $K_3\Box C_4$. Lemma \ref{SporadicGraphs} establishes $q(R_{10,3})=2$. Since $K_3\Box C_4\cong K_3\Box K_2\Box K_2$, Corollary 6.8 in \cite{MR3118943} implies $q(K_3\Box C_4)=2$.  This completes the proof.
\end{proof}

\begin{lemma} \label{HardCase}
If $G$ is a connected $4$-regular graph with $q(G)=2$ and diameter at least $3$, then $G$ is either
\begin{enumerate}[(1)]
\item $K_3\Box C_4$, $K_{3,3}\Box K_2$, one of the graphs $R_{10,2}$, $R_{10,3}$, $R_{10,4}$, $R_{12,3}$, or $R_{14,1}$ given in Figure \ref{RGraphs2};
\item $Q_4$; or,
\item a closed candle $H_k$ for some $k \ge 6$.  
\end{enumerate}
\end{lemma}

\begin{proof}

Recall that a full vertex is a vertex whose incident edges have all been accounted for.

Let $v$ be a vertex of $G$ with $\epsilon(v) \geq 3$.  If no vertex in the distance partition of $v$ has four predecessors, then $G \cong R_{10,3}$ or $G \cong K_3 \Box C_4$ by Lemma \ref{No4pred}.  For the remainder of this proof, we assume that the distance partition of $v$ contains at least one vertex with four predecessors.  Let $d = \epsilon(v)$ be the index of the farthest distance set.  Each vertex with four predecessors uses two of the six extra edges, so $N_d(v)$ may contain no more than three vertices with four predecessors.\\

\noindent {\em Case 1:} Exactly three vertices in $N_d(v)$, say $z_1$, $z_2$, and $z_3$, have four predecessors.

\noindent All six of the extra edges are used by $z_1$, $z_2$, and $z_3$.  By Case 1 of Lemma \ref{Max3LowEdges}, $N_2(v)$ contains exactly six vertices, each being the common successor of a different pair of vertices in $N_1(v)$, and $d=3$.  Since $N_2(v)$ has exactly six vertices with two successors each, $N_3(v) = \{z_1, z_2, z_3\}$ and each vertex in $N_2(v)$ has two neighbors in $\{z_1, z_2, z_3\}$.

Let $N_2(v)=\{y_1, y_2, y_3, y_4, y_5, y_6\}$, and recall from Case 1 of Lemma \ref{Max3LowEdges} that each vertex in $N_2(v)$ shares a common predecessor in $N_1(v)$ with exactly four other vertices in $N_2(v)$.  This partitions $N_2(v)$ into three pairs, say $\{y_1, y_2\}$, $\{y_3, y_4\}$, and $\{y_5, y_6\}$, where each pair does not share a predecessor in $N_1(v)$.  Since each $y_i$ has two successors among $\{z_1, z_2, z_3\}$, each of these pairs must share at least one successor in $N_3(v)$.  Suppose $y_1$ and $y_2$ share $z_1$ as a common successor.  To avoid $y_1z_1y_2$ being a unique path of length 2, $y_1$ and $y_2$ must also share their second successor, say $z_2$.  Similarly, $\{y_3, y_4\}$ and $\{y_5, y_6\}$ each have two shared successors among $\{z_1, z_2, z_3\}$.  Without loss of generality, we may assume that the successors of $\{y_3, y_4\}$ are $\{z_1, z_3\}$ and the successors of $\{y_5, y_6\}$ are $\{z_2, z_3\}$. Thus $G \cong R_{14,1}$, and Lemma \ref{SporadicGraphs} completes the proof.
\\

\noindent {\em Case 2:} Exactly two vertices in $N_d(v)$, say $z_1$ and $z_2$, have four predecessors.

\noindent In this case there remain two extra edges. We treat the possibilities for these edges in the following four subcases.\\

\noindent {\em Case 2(a):} $z_1$ and $z_2$ share all their predecessors in $N_{d-1}(v)$.

\noindent Denote the common predecessors of $z_1$ and $z_2$ by $\{y_1, y_2, y_3, y_4\}$.  Suppose $d> 3$.  Let $\{x_1, x_2\} \subseteq N_{d-2}(v)$ denote the predecessors of $y_1$.  If $x_1$ has three predecessors, then $y_1$ is its unique successor, and $x_1y_1z_1$ is a unique shortest path of length 2 between $x_1$ and $z_1$.  So we may let $\{w_1, w_2\} \subseteq N_{d-3}(v)$ denote the predecessors of $x_1$.  Since $y_1$ is full and $w_1x_1y_1$ is a shortest path of length 2, $w_1$ must be adjacent to $x_2$.  Similarly, $w_2$ is adjacent to $x_2$.  We repeat this argument, with $x_1$ replacing $y_1$, and then with $w_1$ replacing $x_1$, forming successive induced candle sections from $\{x_1, x_2\}$ until we reach, say, $\{a_1, a_2\} \subseteq N_1(v)$.  Since both $x_1$ and $x_2$ are just one edge short of being full, there must exist a vertex, say $y_4$, whose predecessors are not $x_1$ or $x_2$.  Denote the predecessors of $y_4$ by $\{x_3, x_4\}$.  By the same argument, starting with $y_4$ we see successive induced candle sections from $\{x_3,x_4\}$ until $\{a_3, a_4\} \subseteq N_1(v)$.  Recall that $\{x_1, x_2\} \cap \{x_3, x_4\} = \emptyset$.  A vertex in $\{w_1, w_2\} \cap \{w_3, w_4\}$ would have to have four successors, so $\{w_1, w_2\} \cap \{w_3, w_4\} = \emptyset$.  Continuing, the two candle sections must be disjoint, including $\{a_1, a_2\} \cap \{a_3, a_4\} = \emptyset$.  It follows that $N_1(v) = \{a_1, a_2, a_3, a_4\}$.

Suppose the two predecessors of $y_2$ are $x_i\in \{x_1, x_2\}$ and $x_j \in \{x_3, x_4\}$.  Then $x_i, x_j$, and $y_2$ are all full, and $x_iy_2x_j$ is a unique shortest path of length 2 between $x_i$ and $x_j$.  So we assume without loss of generality that the predecessors of $y_2$ are $\{x_1, x_2\}$, and of $y_3$ are $\{x_3, x_4\}$.

Each vertex in $N_1(v)$ is one edge short of being full, and any pair of these vertices shares $v$ as a common neighbor.  Suppose $a_1$ is adjacent to one of $\{a_3, a_4\}$, say $a_3$.  Then $a_1$ and $a_3$ are both full, and $a_1va_4$ is a unique shortest path of length 2 between $a_1$ and $a_4$.  So $a_1$, $a_3$, and $a_4$ must share a common successor $u\in N_2(v)$.  Similarly, $a_2$, $a_3$, and $a_4$ must share a common successor, which must be $u$ since $a_3$ and $a_4$ are full.  Therefore, $u$ has four predecessors, using up the remaining two extra edges, and all vertices listed so far are full.  It follows that $G$ cannot contain any additional vertices, and $G \cong H_k$ for some even $k\geq 6$ (as the number of vertices accounted for in the proof is divisible by $4$).  Recall that $q(H_k)=2$ by Lemma \ref{qCandle}.

Suppose $d=3$ and again write $N_1(v) = \{a_1, a_2, a_3, a_4\}$.  The exact number of edges that connect $\{y_1, y_2, y_3, y_4\}$ to vertices in $N_1(v)$ is eight.  Furthermore, since $z_1$ and $z_2$ are full, each vertex in $N_1(v)$ must have either zero or at least two successors in $\{y_1, y_2, y_3, y_4\}$, to prevent a unique shortest path of length 2 between $N_1(v)$ and $N_3(v)$.  If, say $a_1$, has no successor in $\{y_1, y_2, y_3, y_4\}$, then it must be adjacent to each of $\{a_2, a_3, a_4\}$. However, to account for the edges between $N_1(v)$ and $N_2(v)$, some $a_i$ must have three successors. This is impossible as $\deg(a_i)=4$ for all $i$. So each $a_i$ has two successors in $\{y_1,y_2,y_3, y_4\}$, accounting for eight edges.  These eight edges and $\{y_1, y_2, y_3, y_4\}\cup\{a_1, a_2, a_3, a_4\}$ form $2C_4$ or $C_8$.  In the first case, it follows as in the previous paragraph that $G \cong H_6$.

Suppose the eight edges connecting the vertices $\{y_1, y_2, y_3, y_4\}$ to $\{a_1, a_2, a_3, a_4\}$ form a $C_8$.  Then there are two pairs of vertices in $N_1(v)$ that share $v$ as a predecessor but share no successor in $\{y_1, y_2, y_3, y_4\}$.  Without loss of generality, assume these pairs are $\{a_1, a_3\}$ and $\{a_2, a_4\}$. Suppose the two remaining extra edges are contained in $N_1(v)$. Then the two non-isomorphic ways to place these edges are  $\{a_1a_2, a_3a_4\}$ or $\{a_1a_3, a_2a_4\}$. In the first case, $a_1va_3$ is a unique shortest path of length 2. If $a_1a_3$ and $a_2a_4$ are edges, and if the common successor of $a_1$ and $a_4$ is, say $y_1$, then $y_1a_1a_3$ is a unique shortest path of length 2 between $y_1$ and $a_3$. Therefore $N_1(v)$ cannot contain two edges and there is at least one (and at most two) additional vertex (vertices) in $N_2(v)$.  

In either case, a vertex $y_5 \notin \{y_1, y_2, y_3, y_4\}$ cannot have a successor, since there would be a unique shortest path of length 2 between this successor and the predecessors of $y_5$ in $N_1(v)$.  So $y_5$ must be a terminal vertex. Because we have only two extra edges remaining, $y_5$ must have at most one neighbor in $N_2(v)$ and hence must have at least three predecessors.  If there were a second terminal vertex $y_6\in N_2(v)$, then there would be at least six edges from $\{y_5, y_6\}$ to $N_1(v)$, contradicting that all four vertices in $N_1(v)$ are only one edge short of being full. The only possibility then is for $N_2(v)$ to contain the one additional vertex, $y_5$, with all four vertices in $N_1(v)$ as its predecessors, using up the two remaining extra edges.  Then all vertices are full, and in this case $G \cong R_{12,1}$, which satisfies $q(R_{12,1})>2$, by Lemma \ref{SporadicGraphs}, and leads to a contradiction.\\

\noindent {\em Case 2(b):} $z_1$ and $z_2$ share exactly three of their four predecessors in $N_{d-1}(v)$.

\noindent Denote the predecessors of $z_1$ by $\{y_1, y_2, y_3, y_4\}$ and of $z_2$ by $\{y_2, y_3, y_4, y_5\}$.  Then $y_1$ shares one common successor with $y_2$, $y_3$, and $y_4$.  Given that these three vertices have two predecessors each, they are full, as is their other successor $z_2$. So $y_1$ cannot be adjacent to any of them, nor can it have a second common successor with any of them.  Therefore, $y_1$ must have a common predecessor with each of $y_2$, $y_3$, and $y_4$.  In particular, let $x_1$ be a common predecessor of $y_1$ and $y_2$.  Then the path $x_1y_2z_2$ implies $x_1$ must be adjacent to $y_3$, $y_4$, or $y_5$, which implies $x_1\in N_1(v)$ and $d=3$.  Denote the three remaining vertices in $N_1(v)$ by $\{x_2, x_3, x_4\}$.  Moreover, the fullness of $z_1$, $y_2$, $y_3$, and $y_4$ implies that any neighbor of $y_1$ other than $z_1$ must be its predecessor (to avoid a unique shortest path of length 2 between $z_1$ and this neighbor of $y_1$), and hence $y_1$ has three predecessors in $N_1(v)$.  Similarly, $y_5$ has three predecessors in $N_1(v)$, using up the two remaining extra edges.  

We make two observations.  First, since there are exactly twelve edges connecting $\{y_1, y_2, y_3, y_4, y_5\}$ to vertices in $N_1(v)$, any successor of $x_1$, $x_2$, $x_3$, and $x_4$ must be in $\{y_1, y_2, y_3, y_4, y_5\}$, and
\[
V(G) = \{z_1, z_2, y_1, y_2, y_3, y_4, y_5, x_1, x_2, x_3, x_4, v\}.
\]
Second, $y_1$ and $y_5$ share at least two predecessors.

Suppose $y_1$ and $y_5$ share all three predecessors, say $\{x_1, x_2, x_3\}$.  Then the three successors of $x_4$ must be $y_2$, $y_3$, and $y_4$.  Finally, the remaining predecessor of $y_3$ and $y_4$ must be, without loss of generality, $x_2$ and $x_3$, respectively.  Hence all vertices are full, and $G \cong K_{3,3} \Box K_2$.  Corollaries 6.5 and 6.8 of \cite{MR3118943} show that $q(K_{3,3} \Box K_2)=2$.

Now suppose $y_1$ and $y_5$ share two predecessors, say $\{x_1, x_2\}$.  Without loss of generality, let $x_3$ and $x_4$ be the third predecessors of $y_1$ and $y_5$, respectively.  Either $x_1$ and $x_2$ share their third successor as well or they do not.  Suppose $x_1$ and $x_2$ share all three successors, $\{y_1, y_2, y_5\}$.  Then $y_1$, $y_2$, and $y_5$ are all full, and the four remaining edges between $N_2(v)$ and $N_1(v)$ are determined.  In this case, $G \cong R_{12,2}$, contradicting $q(G) = 2$ by Lemma \ref{SporadicGraphs}.  Therefore, we must have that $x_1$ and $x_2$ do not share a third common successor.  Since the third successor of $x_1$ is $y_2$, let the third successor of $x_2$ be $y_3$.  Again, the four remaining edges between $N_2(v)$ and $N_1(v)$ are determined, up to one choice, which yields graphs that are isomorphic.  In this case, it can be verified that $G \cong R_{12, 3} \cong C(12, \pm1, \pm 3)$, and Lemma \ref{SporadicGraphs} implies $q(R_{12, 3})=2$. \\

\noindent {\em Case 2(c):} $z_1$ and $z_2$ share exactly two of their four predecessors in $N_{d-1}(v)$.

\noindent Suppose the predecessors of $z_1$ are $\{y_1, y_2, y_3, y_4\}$ and the predecessors of $z_2$ are $\{y_3, y_4, y_5, y_6\}$.  With two predecessors, $y_3$ is full (as is $y_4$).  To resolve the unique shortest paths of length 2 through $z_1$ or $z_2$, $y_3$ must share a predecessor with each of $y_1, y_2, y_5$, and $y_6$.  It follows that the predecessors of $y_3$ have three successors each, and therefore $d=3$.

Since there are at least twelve edges from $\{y_1, y_2, y_3, y_4, y_5, y_6\}$ to $N_1(v)$, and at most twelve edges from $N_1(v)$ to $N_2(v)$, we must have: $N_2(v) = \{y_1, y_2, y_3, y_4, y_5, y_6\}$, each vertex in $N_2(v)$ has exactly two predecessors in $N_1(v)$, and each vertex in $N_1(v)$ has exactly three successors in $N_2(v)$.  Note that $y_3$ and $y_4$ are full while $y_1$, $y_2$, $y_5$, and $y_6$ are each one edge short of full. We label $N_1(v)=\{x_1, x_2, x_3, x_4\}$. 

To prevent unique shortest paths of length 2 between $x_i$ and $z_j$, each $x_i$ must be adjacent to at least one vertex in $\{y_3, y_4\}$. It follows that each $x_i$ must have exactly one neighbor in $\{y_3, y_4\}$.  Assume one successor of $x_1$ is $y_3$.  To avoid unique paths of length 2 between $x_1$ and both $z_1$ and $z_2$, $x_1$ must have at least one more successor in $\{y_1, y_2\}$ and at least one more successor in $\{y_5, y_6\}$. Assume, without loss of generality, that the successors of $x_1$ are $\{y_1, y_3, y_5\}$.  Suppose $y_1$ has a successor, $z_3$, in addition to $z_1$.  Then $z_3$ must be adjacent to $y_5$ (to prevent $x_1y_1z_3$ being a unique shortest path of length 2 between $x_1$ and $z_3$) and $z_3$ must be adjacent to $y_6$ (to prevent $z_3y_5z_2$ being a unique shortest path of length 2 between $z_3$ and $z_2$).  This leaves all vertices in $N_2(v)$ except $y_2$ full, but any successor of $y_2$ distinct from $z_1$, $z_2$, and $z_3$ would create a unique shortest path of length 2 between this successor and a predecessor of $y_2$. So the successor of $y_2$ must be $z_3$, contradicting the assumption of Case 2.  

To become full, $y_1$ must have exactly one neighbor in $\{y_2, y_5, y_6\}$.  If $y_1$ is adjacent to $y_2$, then $x_1y_1y_2$ is a unique shortest path of length 2 between the full vertices $x_1$ and $y_2$.  But if $y_1$ is adjacent to $y_5$ or $y_6$, then $z_1y_1y_5$ or $z_1y_1y_6$ is a unique shortest path of length 2 between $z_1$ and $y_5$ or $y_6$, respectively, which would contradict $q(G)=2$.\\

\noindent {\em Case 2(d):} $z_1$ and $z_2$ share exactly one, or none, of their four predecessors in $N_{d-1}(v)$.

\noindent Here $N_{d-1}(v)$ must contain seven or eight vertices, at least six of which, after accounting for two predecessors, are exactly one edge short of being full.  If any such $y\in N_{d-1}(v)$ has another successor $z_3 \in N_d(v)$, then $z_3$ must have a neighbor $z_4 \in N_d(v)\setminus \{z_1, z_2\}$. Since $y$, $z_1$, and $z_2$ are full, $yz_3z_4$ is a unique shortest path of length 2 between $y$ and $z_4$. In order for the vertices in $N_{d-1}(v)$ to become full, at least six vertices require incidence with one of the two remaining extra edges, which is impossible.\\

\noindent {\em Case 3:} Exactly one vertex, say $z$, in $N_d(v)$ has four predecessors.

\noindent Four extra edges remain.  Denote the predecessors of $z$ by $\{y_1, y_2, y_3, y_4\} \subseteq N_{d-1}(v)$. \\

\noindent {\em Case 3(a):} $y_1$, $y_2$, $y_3$, and $y_4$ each have three predecessors.

\noindent All remaining extra edges are used, and $N_d(v) = \{z\}$ and $N_{d-1}(v) = \{y_1, y_2, y_3, y_4\}$.  Suppose $d=3$.  Each vertex in $N_1(v)$ has three successors in $N_2(v)$ (to account for the twelve edges connecting $N_2(v)$ to $N_1(v)$), and each $y_i$ has one vertex in $N_1(v)$ as its non-neighbor.  If $y_i\neq y_j$ have the same non-neighbor, then that vertex cannot have three successors.  So the four vertices in $N_2(v)$ are in bijection with the four vertices in $N_1(v)$, via non-neighbors, determining the graph as $G\cong R_{10,4}$.  Lemma \ref{SporadicGraphs} shows $q(R_{10,4})=2$.

Suppose $d>3$.  Each pair of vertices in $\{y_1, y_2, y_3, y_4\}$ shares $z$ as a common successor, and each $y_i$ is full with three predecessors, so each pair must share a common predecessor in $N_{d-2}(v)$.  The vertices in $N_{d-2}(v)$ must have two predecessors each and therefore can have at most two successors in $N_{d-1}(v)$.  This forces $N_{d-2}(v)$ to contain six vertices $\{x_1, x_2, x_3, x_4, x_5, x_6\}$, each being the common predecessor of a different pair of vertices in $\{y_1, y_2, y_3, y_4\}$, exhausting all edges between $\{y_1, y_2, y_3, y_4\}$ and $\{x_1, x_2, x_3, x_4, x_5, x_6\}$.

Consider $y_1$ and, without loss of generality, let $\{x_1, x_2, x_3\}$ denote its three predecessors.  Then $x_4$, $x_5$, and $x_6$ each must have their two successors in $\{y_2, y_3, y_4\}$.  We claim that $x_4$, $x_5$, and $x_6$ must have a common predecessor in $N_{d-3}(v)$.  Note that if $x_4$, $x_5$ and $x_6$ have a common predecessor, this predecessor must lie in $N_1(v)$, and $d=4$. Indeed, if any two of $\{x_4, x_5, x_6\}$ share a predecessor, then all three must share a predecessor, in order to prevent a unique shortest path of length 2 between this predecessor and one of $\{y_2, y_3, y_4\}$.  However, if $x_4$, $x_5$, and $x_6$ have three distinct predecessors in $N_{d-3}(v)$ that are not predecessors of either of the other two, then each of these three predecessors must have two successors in $\{x_1, x_2, x_3\}$, thereby causing $x_1$, $x_2$, and $x_3$ to be full; the fourth vertex of $N_{d-3}(v) = N_1(v)$ must then have $x_4$, $x_5$, and $x_6$ as its three successors.

Let $w_1$ be a shared predecessor of $\{x_4, x_5, x_6\}$.  Then there are two paths of length 2 between $w_1$ and each of $\{y_2, y_3, y_4\}$ and no paths of length 2 between $w_1$ and $y_1$.  Repeat this procedure for $y_2$, $y_3,$ and $y_4$, each time finding a common predecessor $w_i$ of
\[
\{x_1, x_2, x_3, x_4, x_5, x_6\} \setminus \{\mbox{predecessors of } y_i\}
\]
for $i=2, 3, 4$.  The choice of edges between $\{y_1, y_2, y_3, y_4\}$ and $\{x_1, x_2, x_3, x_4, x_5, x_6\}$ therefore determines the edges between $\{x_1, x_2, x_3, x_4, x_5, x_6\}$ and $\{w_1, w_2, w_3, w_4\}$. However, since all pairs of vertices in $\{y_1, y_2, y_3, y_4\}$ share a predecessor among the independent set $\{x_1, x_2, x_3, x_4, x_5, x_6\}$ of vertices, all choices yield isomorphic graphs.  All vertices are now full, so we have listed all elements of each distance set, and $G \cong Q_4$.  By \cite{MR3118943} Corollary 6.9, $q(Q_4) = 2$.\\

\noindent {\em Case 3(b):} There exists a vertex in $N_{d-1}(v)$, say $y_1$, that does not have three predecessors.

\noindent The edges incident with $z$ account for two of the extra edges. If $y_1$ has another successor $z_2\in N_d(v)$, then $z_2$ must have a neighbor $z_3\in N_d(v) \setminus \{z, z_2\}$; since $y_1$ (with two predecessors) and $z$ are both full, $y_1z_2z_3$ is a unique shortest path of length 2 between $y_1$ and $z_3$. So $y_1$ must have a neighbor $u \in N_{d-1}(v)$, using up an extra edge.

Suppose $u\notin \{y_2, y_3, y_4\}$.  Then $u$ must be adjacent to another vertex in $\{y_1, y_2, y_3, y_4\}$, say $y_2$, to avoid a unique shortest path of length 2 between itself and $z$.  This leaves two remaining extra edges.  Let $x_1$ and $x_2$ denote the two predecessors of $u$.  Suppose $x_1$ has two predecessors.  Then it cannot be adjacent to both $y_1$ and $y_2$.  So there must exist $x_3 \in N_{d-2}(v)$ (possibly $x_3=x_2$) such that $x_1$ is adjacent to $x_3$, and $x_3$ is adjacent to $y_1$ and $y_2$.  But then $x_3$ can have only one predecessor, so we must have $d=3$.

Let $N_1(v) = \{x_1, x_2, x_3, x_4\}$. If $x_1$ and $x_2$ are not incident with an edge of $G[N_1(v)]$, then $x_1$ and $x_2$ both must be adjacent to $\{u, y_1, y_2\}$ (to avoid uniqueness of the shortest paths between each $x_1$, $x_2$ and each $y_1$, $y_2$ through $u$); but then $y_1$ and $y_2$ are full, and $y_1zy_3$ is a unique shortest path of length 2 between $y_1$ and $y_3$.  Hence $G[N_1(v)]$ must contain an edge, using a third extra edge.  It follows that there are exactly ten edges connecting $N_1(v)$ to $N_2(v)$.  Then no vertex in $N_2(v)$ can have three predecessors, which implies $G[N_2(v)]$ must contain another edge, the last of the extra edges.  Moreover, this edge must connect $y_3$ and $y_4$. Note this implies $N_1(v)$ cannot be an independent set of vertices.

Suppose $x_2$ is incident with the edge in $G[N_1(v)]$, and $x_1$ is not.  Then $x_1$ must be adjacent to $\{u, y_1, y_2\}$.  And since $x_1$ shares $v$ as a common predecessor with $x_3$ and $x_4$, and since $u$ is full while $y_1$ and $y_2$ are one edge short of being full, we must have, without loss of generality, $y_1$ adjacent to $x_3$ and $y_2$ adjacent to $x_4$.  Since $x_1$ and $y_1$ are full, the uniqueness of the path $x_2uy_1$ can only be avoided if $x_2$ is adjacent to $x_3$.  Similarly, the uniqueness of the path $x_2uy_2$ can only be avoided if $x_2$ is adjacent to $x_4$.  But $G[N_1(v)]$ can only have one edge.

So the edge in $G[N_1(v)]$ must be $x_1x_2$. Then $x_1$ has only one more successor in $N_2(v)$ which must simultaneously resolve $v$ being the unique common neighbor of the pairs $\{x_1, x_3\}$ and $\{x_1, x_4\}$.  This successor must then have three predecessors, which is impossible. Therefore $u \in \{y_2, y_3, y_4\}$.

Without loss of generality, suppose $u=y_2$.  As at the start of this case, $\{y_1, y_2, y_3, y_4\}$ can have no more successors in $N_d(v)$.  And $y_1$ and $y_2$ cannot be incident to another edge in $G[N_{d-1}(v)]$, in order to allow for each to have two predecessors.  To prevent $y_1$ and $y_3$ having $z$ as a unique common neighbor, they must therefore share a common predecessor, say $x_1\in N_{d-2}(v)$.  If $x_1$ has two predecessors, then it is full, and $x_1y_1y_2$ is a unique shortest path of length 2 between $x_1$ and $y_2$.  So $x_1$ must have only one predecessor, which implies $d=3$.

Suppose that $G[N_2(v)]$ contains no edges other than $y_1y_2$; then $y_3$ and $y_4$ must have three predecessors each.  Denote the other vertices in $N_1(v)$ by $x_2$, $x_3$, and $x_4$.  Now, $y_3$ and $y_4$ must share at least two of their predecessors in $N_1(v)$, but $x_1$ cannot be this shared predecessor, since that would leave $x_1y_1y_2$ as the unique shortest path of length 2 between $x_1$ and $y_2$.  So assume the two shared predecessors of $y_3$ and $y_4$ are $\{x_3, x_4\}$.  Neither $x_3$ nor $x_4$ can be adjacent to $y_1$ or $y_2$, as these adjacencies would result in a unique shortest path of length 2 between one of $\{x_3, x_4\}$ and one of $\{y_1, y_2\}$.  With exactly one remaining extra edge, this edge must be $x_3x_4$ to ensure that $x_3$ and $x_4$ each have four neighbors.  It now follows that $x_1$ is adjacent to $y_2$, and hence the three remaining edges are $x_2y_1$, $x_2y_2$, and $x_2y_4$.  After connecting each $x_i$ to $v$, all vertices are full, and $G\cong R_{10,3}$ (to see this let $v$ be vertex $3$ in the drawing of $R_{10,3}$ in Figure \ref{RGraphs2}). Applying Lemma \ref{SporadicGraphs} establishes $q(R_{10,3})=2$.

Finally, suppose $G[N_2(v)]$ contains another edge, which can only be $y_3y_4$, as otherwise $y_1$ or $y_2$ would not have enough predecessors. So each $y_i$ has exactly two predecessors.  Then there are eight edges between $N_2(v)$ and $N_1(v)$, and each vertex in $N_1(v)$ must have exactly two successors.  Each pair $\{y_1, y_3\}$, $\{y_1, y_4\}$, $\{y_2, y_3\}$, and $\{y_2, y_4\}$ must share a single predecessor in $N_1(v)$, causing each $y_i$ to be full.  Denote these unique common predecessors by $x_1$, $x_2$, $x_3$, and $x_4$, respectively.  To simultaneously avoid the unique shortest paths of length 2 from $x_1$ to $y_2$ and $y_4$, $x_1$ must be adjacent to $x_4$.  Similarly, $x_2$ must be adjacent to $x_3$.  This uses up the remaining extra edges.  After connecting each $x_i$ to $v$, all vertices are full, and $G \cong R_{10,2}$. Using Lemma \ref{SporadicGraphs} establishes $q(R_{10,2})=2$.\\

\noindent {\em Case 4:} $N_d(v)$ contains no vertices with four predecessors.

\noindent Since the distance partition from $v$ contains at least one vertex with four predecessors, let $b$ denote a vertex with four predecessors.  This accounts for two of the six extra edges.  By the proofs of Lemmas \ref{Max3LowEdges} and \ref{No4pred}, $N_d(v)$ must contain a pair of vertices with three predecessors each or a cycle of vertices with two predecessors each. Furthermore, each edge in $N_d(v)$ is an extra edge and any vertex in $N_d(v)$ with three predecessors accounts for an extra edge, and so $G[N_d(v)]$ must be isomorphic to one of $K_2$, $P_3$, $C_3$, or $C_4$. Each of these options requires either three or four of the remaining extra edges.  Therefore, there is at most one extra edge unaccounted for.  This implies $b$ is the unique vertex with four predecessors. 

We first claim that $b \in N_2(v)$.  Indeed, suppose $b\in N_i(v)$ for some $2<i<d$, and let $\{a_1, a_2, a_3, a_4\} \subseteq N_{i-1}(v)$ be the four predecessors of $b$.  Then each $a_i$ has at least two predecessors.  With at most one of the extra edges left, we may assume, without loss of generality, that $a_1$ has a successor $b_1\in N_i(v)$.  Now, $b_1$ cannot be a terminal vertex, as that would require more extra edges than are currently unaccounted for, so let $c_1 \in N_{i+1}(v)$ be a successor of $b_1$.  Then $a_1b_1c_1$ is a unique shortest path of length 2.  So $b\in N_2(v)$.

We next claim that there are no vertices with three predecessors in $N_i(v)$ for any $i<d$.  As observed at the end of Claim \#1 in the proof of Lemma \ref{No4pred}, any successor of a vertex with three predecessors must be a vertex with four predecessors.  Since $b\in N_2(v)$, any vertex with three predecessors in $N_i(v)$ for $i\geq 2$ cannot have a successor.  It follows that such a vertex requires an edge in $G[N_i(v)]$, which would require the use of at least three extra edges.  A similar argument shows that any vertex in $N_i(v)$ for $2\leq i < d-1$ cannot have a unique successor in $N_{i+1}(v)$. If there were a vertex $x$ with two predecessors and a unique successor $y$, it must be adjacent to another vertex in $N_i(v)$. Then $y$ must either be adjacent to a vertex in $N_{i+2}(v)$ or $N_{i+1}(v)$ yielding either a unique shortest path or too many extra edges. So, with the exception of $b$, every vertex in $N_i(v)$ for $2\leq i < d$ has exactly two predecessors, and every vertex in $N_i(v)$ for $2\leq i < d-1$ has exactly two successors.

We have $\{a_1, a_2, a_3, a_4\} = N_1(v)$, and $b$ is a common successor of these four vertices.  With only one remaining extra edge, each $a_i$ must have another successor in $N_2(v)$.  Let $b_1 \neq b$ be a successor of $a_1$ in $N_2(v)$, let $a_2$ be the second predecessor of $b_1$, and let $c_1 \in N_3(v)$ be a successor of $b_1$. Note that this implies $a_1$ and $a_2$ are not adjacent. To avoid unique paths of length 2 between $c_1$ and each of $a_1$ and $a_2$, the second predecessor of $c_1$, say $b_2$, must have $a_1$ and $a_2$ as its predecessors.  Continuing in this way through $N_i(v)$ for all $3\leq i\leq d-2$, we see successive induced candle sections starting with $\{a_1, a_2\}$, and ending say at $\{y_1, y_2\}\subseteq N_{d-1}(v)$. Let $z_1$ be the common successor of $\{y_1, y_2\}$ in $N_d(v)$. This argument is similar to the argument we used in Case 2(a). Starting again from $a_3$, we see successive induced candle sections starting with $\{a_3, a_4\}$ and ending at say $\{y_3, y_4\}\subseteq N_{d-1}(v)$ whose common successor is $z_2\in N_d(v)$.

Note that all vertices in $N_i(v)$ for $i<d-1$ are full, which implies $N_{d-1}(v) = \{y_1, y_2, y_3, y_4\}$. Each $N_i(v)$ is independent for $i<d-1$, and we claim that $N_{d-1}(v)$ is independent as well. To see this, $y_iy_j$ an edge for $i\in\{1,2\}$ and $j\in\{3,4\}$ would create a unique shortest path of length 2 from $N_{d-1}(v)$ to $N_{d-2}(v)$. If $y_1y_2$ is an edge, then $y_1$ and $y_2$ are full and hence $y_1z_1z$ is a unique shortest path for any neighbor $z$ of $z_1$ in $N_d(v)$. The same observation shows $y_3y_4$ cannot be an edge. Hence the last remaining extra edge must be in $G[N_d(v)]$. This implies that $G[N_d(v)]$ is either isomorphic to $P_3$ or $C_4$. Note that if some vertex in $N_d(v)$ has three predecessors, then without loss of generality, it is adjacent to $y_1$, $y_2$, and $y_3$. This results in a unique shortest path of length 2 through $y_3$. Thus $G[N_d(v)]\cong C_4$.

We now claim that any successor of $y_1$ or $y_2$ must be a successor of both. Indeed, if $z$ is a successor of $y_1$, then the paths $zy_1x_1$ and $zy_1x_2$ force $z$ to be adjacent to $y_2$ as well. Similarly, any successor of $y_3$ or $y_4$ must be a successor of both. Hence we have four vertices in $N_d(v)$ forming a cycle $C_4$ where each vertex is adjacent to either $\{y_1, y_2\}$ or $\{y_3, y_4\}$. This implies that the vertices around the cycle must alternate between being adjacent to $\{y_1, y_2\}$ and $\{y_3, y_4\}$ and so $G \cong H_k$ for some odd $k\geq 7$ (as the size of $V(G)\setminus\{b,v\}$ is divisible by 4). Applying Lemma \ref{qCandle} establishes $q(H_k)=2$.

\end{proof}

This completes the proof of Theorem \ref{Regular4}.

\section{Further Observations and Related Problems}

Corollary \ref{Regular123} and Theorem \ref{Regular4} give a complete characterization of the $r$-regular graphs that admit two distinct eigenvalues for $r\leq 4$. For $r=4$, the closed candles $H_k$ for $k\geq 3$ give the only infinite family of $4$-regular graphs $G$ with $q(G)=2$. From these graphs we can construct an infinite family of $r$-regular graphs $G$ with $q(G)=2$ for all $r\geq 5$. 

From Corollary 6.8 in \cite{MR3118943}, if $G$ is any graph with $q(G)=2$, then $q(G\Box K_2)=2$. Define $H_k^d=H_k\Box Q_d$ for any $d\geq 0$ and $k\geq 3$. Now from Lemma \ref{qCandle} and \cite[Cor. 6.9]{MR3118943} we immediately have the following result.
\begin{proposition}\label{RegFamily}
For all $k\geq 3$ and $d\geq 0$, $H_k^d$ is a $(4+d)$-regular graph with $q(H_k^d)=2$.
\end{proposition}

When $r\geq 5$, the number of edges in an $r$-regular graph grows faster than the lower bound $|E(G)|\geq 2n-4$. So, characterizing $r$-regular graphs with two distinct eigenvalues for $r\geq 5$ could be difficult using the methods in this paper. This motivates questions about existence and structure of such graphs. That is, are there other infinite families of $r$-regular graphs $G$ with $q(G)=2$, and if so, can they be constructed with a method other than the Cartesian product used in Proposition \ref{RegFamily}?
%
\begin{problem}\label{RegularProblem}
Determine whether or not an infinite family of $5$-regular graphs $G$ with $q(G)=2$ disjoint from the family $\{H_k^1\,:\,k\geq 3\}$ exists.
\end{problem}

This paper builds on the work in \cite{AllowsProblem}, where we established a lower bound on the number of edges a graph on $n$ vertices must have in order to have $q(G)=2$, and characterized the graphs that meet the bound with equality. The graphs that meet the bound $|E(G)|\geq 2n-4$ are $Q_3$ and an infinite family of graphs called \emph{double-ended candles}. The lower bound is improved to $|E(G)|\geq 2n-3$ if $G$ has an odd number of vertices, and the graphs that meet this improved bound are exactly the infinite family of graphs called \emph{single-ended candles}. These graphs, along with the closed candles $H_k$ and some of the sporadic graphs in Theorem \ref{Regular4}, appear in a set of papers focusing on a different problem.

McKee and Smyth \cite{MR2360149}, Taylor \cite{MR2774674}, and Greaves \cite{MR2954487, MR2990027} all consider the following problem: for which graphs $G$ is there a matrix \emph{compatible} with $G$ whose eigenvalues are contained in the interval $[-2,2]$? For this problem, ``compatible'' means that the matrix $M$ has exactly the same zero pattern as $A(G)$ off of the diagonal entries, and the entries of $M$ are restricted to lie in the ring of integers for some quadratic field (McKee and Smyth consider matrices with integer entries. Taylor and Greaves consider matrices whose entries lie in the ring of integers for an imaginary quadratic extension of $\mathbb{Q}$). It follows from interlacing that if $G$ has a compatible matrix whose eigenvalues are contained in $[-2,2]$, then so do all of its induced subgraphs. This motivates the characterization of the maximal graphs $G$ for which there is a matrix compatible with $G$ whose eigenvalues are contained in $[-2,2]$. Each of the papers \cite{MR2360149},  \cite{MR2954487}, \cite{MR2990027}, and \cite{MR2774674} gives a characterization for the respective rings under consideration. The graphs that appear in these characterizations are a collection of small sporadic graphs, and a few infinite families: the single-ended candles, the double-ended candles, the closed candles, and the graphs obtained from a candle section by adding an edge between each pair of nearest degree $2$ vertices.

We are currently unable to explain the overlap between the graphs in our characterizations of graphs with $2n-4$, $2n-3$, or $2n$ edges that admit two distinct eigenvalues, and the graphs in the characterizations in \cite{MR2360149},  \cite{MR2954487}, \cite{MR2990027}, and \cite{MR2774674}. It seems reasonable to expect that the maximal graphs whose compatible matrices have all of their eigenvalues contained in $[-2,2]$ should have all of their eigenvalues contained in $\{-2,2\}$ and hence have $q$-value $2$. But we do not have a proof of this claim. A complete explanation of this coincidence may help the investigation of $r$-regular graphs $G$ with $q(G)=2$ for $r\geq 5$.
\begin{problem}\label{MaximalProblem}
For a given integer $l\geq 3$, determine whether the maximal graphs that admit an integer (or similarly constrained) matrix whose eigenvalues are contained in $[-l,l]$ have $q(G)=2$.
\end{problem}

A regular graph $G$ that is neither complete nor empty is \emph{strongly regular} if there exist constants $a$ and $c$ so that any two adjacent vertices in $G$ have $a$ common neighbors, and any two non-adjacent vertices have $c$ common neighbors.  The definition implies that the diameter of a connected strongly regular graph is at most $2$.  If $G$ is a connected regular graph, then $A(G)$ has three distinct eigenvalues if and only if $G$ is a strongly regular graph (see, e.g., Lemma 10.2.1 in \cite{MR1829620}). 
So $q(G)\in\{2,3\}$ for a strongly regular graph $G$. In \cite{MR4214542}, Furst and Grotts show that $L(K_n)$, the line graph of $K_n$, has $q(L(K_n))=2$ for all $n\geq 3$ (note that $L(K_n)$ is a strongly regular graph for all $n\geq 3$).


A connected graph $G$ is \emph{distance regular} if there are integers $b_i$ and $c_i$ for all $i\geq 0$ so that for any two vertices $u$ and $v$ at distance $i$ in $G$,
\[
|N_1(v)\cap N_{i+1}(u)|=b_i\quad\text{and}\quad|N_1(v)\cap N_{i-1}(u)|=c_i.
\]
(See \cite{BCN} for an extensive treatment of distance-regular graphs.) Distance-regular graphs with diameter $2$ are strongly regular. Since Corollary \ref{Regular123} and Theorem \ref{Regular4} give a complete characterization of the $r$-regular graphs that admit two distinct eigenvalues for $r\leq 4$, they also determine the distance-regular graphs of valency at most $4$ that admit two distinct eigenvalues. Excluding the complete graphs and complete multipartite graphs in the list, these are $Q_3$, $R_{10,4}$, $R_{14,1}$, and $Q_4$. Note that these graphs all have diameter at least $3$, and thus are not strongly regular. The closed candles $H_3$ and $H_4$ are complete multipartite graphs, and all $H_k$ for $k\geq 5$ are not distance regular.
\begin{problem}\label{SRGProblem}
Determine the distance-regular graphs $G$ with $q(G)=2$.
\end{problem}

We finish with two observations about the $4$-regular graphs $G$ with $q(G)=2$, and problems they raise. With the exception of $K_5$ and $R_{7,1}$, all of the $4$-regular graphs that appear in Theorem \ref{Regular4} have even order. In particular, if $|V(G)|\geq 8$, then $G$ must have even order.
\begin{problem}\label{EvenProblem}
Determine whether for every integer $k$, there is some value $n_k$ so that every $k$-regular graph $G$ with $|V(G)|\geq n_k$ and $q(G)=2$ has even order.
\end{problem}

The graphs $H_k$ are all circulant graphs, as are some of the sporadic graphs listed in Theorem \ref{Regular4}. More broadly, many of the graphs in our characterization are vertex transitive. It would be interesting to investigate the highly symmetric graphs that admit a matrix with two distinct eigenvalues. 
\begin{problem}\label{CirculantProblem}
Determine the circulant graphs $G$ that have $q(G)=2$. What can be said about the automorphism groups of regular graphs with $q(G)=2$? 
\end{problem}


\section*{Acknowledgements}
Shaun M. Fallat was supported in part by an NSERC Discovery Research Grant, Application No.: RGPIN--2019--03934.  Veronika Furst was supported in part by a Fort Lewis College Faculty Development Grant.  Shahla Nasserasr was supported in part by a Rochester Institute of Technology COS Dean’s Research Initiation Grant. Brendan Rooney was supported in part by an RIT COS Dean’s Research Initiation Grant. Michael Tait was supported in part by NSF grant DMS-2011553 and a Villanova University Summer Grant.

This project began as part of the ``Inverse Eigenvalue Problems for Graphs and Zero Forcing” Research Community sponsored by the American Institute of Mathematics (AIM). We thank AIM for their support, and we thank the organizers and participants for contributing to this stimulating research experience.

We thank Tracy Hall for the matrix associated with the graph $R_{7,1}$ and Bryan Shader for the matrices used in the proof of Lemma \ref{qCandle}.

\bibliographystyle{plain}
\bibliography{bibliography}

\end{document}